\documentclass[letterpaper, 10 pt, conference]{ieeeconf}  

\usepackage{amsmath,amscd}
\usepackage{bbm}

\usepackage{dsfont}
\usepackage{bm}
\usepackage{amsthm,array}
\usepackage{amsmath}
\usepackage{amsfonts,latexsym}
\usepackage{graphicx,wrapfig}
\usepackage{subfig}
\usepackage{times}
\usepackage{psfrag,epsfig}
\usepackage{verbatim}
\usepackage{tabularx}
\usepackage[hidelinks]{hyperref}
\usepackage{color}
\usepackage{soul}
\usepackage{indentfirst}
\usepackage{inputenc}
\usepackage[capitalise]{cleveref}
\usepackage{courier}
\usepackage{authblk}
\usepackage{mathtools}
\usepackage{mathtools}
\usepackage{cite}
\usepackage{amssymb}
\usepackage{graphicx}
\usepackage{caption}
\usepackage{dblfloatfix}
\usepackage{lineno}
\usepackage{multirow}
\usepackage{algorithm,algpseudocode}
\usepackage{ltablex}
\usepackage{relsize}
\let\labelindent\relax
\usepackage[shortlabels]{enumitem}
\usepackage{xtab}
\usepackage{xfrac}

\xentrystretch{-0.07}

\makeatletter
\def\mathcenterto#1#2{\mathclap{\phantom{#1}\mathclap{#2}}\phantom{#1}}
\let\old@widetilde\widetilde
\def\widetildeto#1#2{\mathcenterto{#2}{\old@widetilde{\mathcenterto{#1}{#2\,}}}}
\let\old@widehat\widehat
\def\widehatto#1#2{\mathcenterto{#2}{\old@widehat{\mathcenterto{#1}{#2\,}}}}
\let\old@overline\overline
\def\overlineto#1#2{\mathcenterto{#2}{\old@overline{\mathcenterto{#1}{#2\,}}}}
\makeatother

\DeclareMathOperator*{\trace}{\mathrm{trace}}

\def\mm{\mathsf{MM-PGO}}
\def\amm{\mathsf{AMM-PGO}}
\def\nag{\mathsf{AMM-Chordal}}
\def\dgs{\mathsf{DGS}}

\def\0{\boldsymbol{0}}

\def\b{\boldsymbol{b}}

\def\NN{\mathcal{N}}

\def\I{\mathbf{I}}

\def\transpose{\top}

\def\R{\mathbb{R}}

\def\b0{\mathbf{0}}

\def\AA{\mathcal{L}}
\def\AA{\mathcal{A}}
\def\XX{\mathcal{X}}

\def\RR{\mathcal{R}}
\def\TT{\mathcal{T}}

\def\se3{\mathfrak{se}(3)}

\def\diag{\mathrm{diag}}

\def\nR{{\widetilde{R}}}
\def\nt{\tilde{t}}
\def\nM{\widetildeto{H}{M}}

\def\nH{{\widetilde{\mathrm{\Omega}}}}

\def\nGamma{\widetilde{\mathrm{\Gamma}}}
\def\nR{{\widetilde{R}}}
\def\nt{\tilde{t}}

\def\QQ{\mathcal{Q}}

\def\EE{\mathcal{E}}
\def\aEE{\overrightarrow{\EE}}

\long\def\answer#1{}

\long\def\comment#1{}

\def\grad{\mathrm{grad}}

\newcommand{\innprod}[2]{\langle {#1},{#2}\rangle}

\theoremstyle{definition}

\newtheorem{prop}{Proposition}

\newtheorem{assumption}{Assumption}

\crefname{prop}{Proposition}{Propositions}
\crefname{assumption}{Assumption}{Assumptions}

\theoremstyle{remark}
\newtheorem*{remark*}{{Remark}}

\newcolumntype{P}[1]{>{\centering\arraybackslash}p{#1}}
\newcolumntype{M}[1]{>{\centering\arraybackslash}m{#1}}

\newcommand{\RNum}[1]{\uppercase\expandafter{\romannumeral #1\relax}}

\hypersetup{
    bookmarks=true,         
    unicode=false,          
    pdftoolbar=true,        
    pdfmenubar=true,        
    pdffitwindow=false,     
    pdfstartview={FitH},    
    pdftitle={My title},    
    pdfauthor={Author},     
    pdfsubject={Subject},   
    pdfcreator={Creator},   
    pdfproducer={Producer}, 
    pdfkeywords={keyword1} {key2} {key3}, 
    pdfnewwindow=true,      
    colorlinks=true,       
    linkcolor=blue,          
    citecolor=red,        
    filecolor=magenta,      
    urlcolor=cyan           
}

\IEEEoverridecommandlockouts                              

\title{\LARGE \bf  Majorization Minimization Methods for Distributed Pose Graph Optimization with Convergence Guarantees}

\author{Taosha Fan \and Todd Murphey
\thanks{Taosha Fan and Todd Murphey are with the Department of Mechanical Engineering, Northwestern University, Evanston, IL 60201, USA. E-mail: {\tt taosha.fan@u.northwestern.edu, t-murphey@northwestern.edu}\\\indent This material is based upon work supported by the National Science Foundation under award DCSD-1662233.}
    }

\begin{document}
\maketitle
\thispagestyle{empty}
\pagestyle{empty}

\begin{abstract}
	In this paper, we consider the problem of distributed pose graph optimization (PGO) that has extensive applications in multi-robot simultaneous localization and mapping (SLAM). We propose majorization minimization methods for distributed PGO and show that our proposed methods are guaranteed to converge to first-order critical points under mild conditions. Furthermore, since our proposed methods rely a proximal operator of distributed PGO, the convergence rate can be significantly accelerated with Nesterov's method, and more importantly, the acceleration induces no compromise of theoretical guarantees. In addition, we also present accelerated majorization minimization methods for the distributed chordal initialization that have a quadratic convergence, which can be used to compute an initial guess for distributed PGO. The efficacy of this work is validated through applications on a number of 2D and 3D SLAM datasets and comparisons with existing state-of-the-art methods, which indicates that our proposed methods have faster convergence and result in better solutions to distributed PGO. 
\end{abstract}

\section{Introduction} \label{section::introduction}
In multi-robot simultaneous localization and mapping (SLAM) \cite{cunningham2010ddf,dong2015distributed,cieslewski2018data}, each robot needs to estimate its own poses from noisy relative pose measurements under limited communication with the other robots, and such a problem can be formulated as distributed pose graph optimization (PGO), in which each robot can be represented as a node and two nodes (robots) are said to be neighbors if there exists a noisy relative pose measurement between them. In most cases, it is assumed that each node can only communicate with its neighbors, which suggests that distributed PGO is much more challenging than non-distributed PGO. Even though there exists numerous methods for PGO \cite{rosen2016se,dellaert2012factor,carlone2016planar,carlone2015lagrangian,kuemmerle11icra,grisetti2009nonlinear,fan2019cpl}, most, if not all, of them are difficult to be distributed. 

In the last decade, multi-robot SLAM has been becoming increasingly important, which further promotes the research of distributed PGO \cite{choudhary2017distributed,choudhary2015exactly,tron2014distributed}. A distributed linear system solver is implemented in \cite{choudhary2017distributed} to evaluate the chordal initialization \cite{carlone2015initialization} and the Gauss-Newton direction for distributed PGO. A heuristic extension of the alternating direction method of multipliers (ADMM) for distributed PGO is proposed in \cite{choudhary2015exactly}. A multi-stage first-order method for distributed PGO using the Riemannian gradient is presented in \cite{tron2014distributed}. However, \cite{choudhary2017distributed,choudhary2015exactly} have no guarantees to converge to critical points, whereas \cite{tron2014distributed} needs strict presumptions that might fail to hold in practice.

In this paper, we propose majorization minimization methods \cite{hunter2004tutorial} for distributed PGO that extend our previous work \cite{fan2019proximal}, in which proximal methods for PGO are proposed that converge to first-order critical points for both centralized and distributed PGO. In \cite{fan2019proximal}, each pose is represented as a single node and updated independently. Even though proximal methods for PGO in \cite{fan2019proximal} converge fast for centralized PGO and apply to any distributed PGO, it might have slow convergence for multi-robot SLAM, in which each robot usually has more than one poses and it is more reasonable to represent poses of the same robot rather than each individual pose as a node. In this paper, poses of the same robot are represented as the same node and are updated as a whole, which makes use of more information in optimization and is expected to converge faster than \cite{fan2019proximal} for distributed PGO in multi-robot SLAM. Furthermore, we redesign the accelerated algorithm using Nesterov's method for distributed PGO such that the inefficient objective evaluation is avoided and the inter-node communication is significantly reduced.

In contrast to existing state-of-the-art methods \cite{choudhary2017distributed,choudhary2015exactly,tron2014distributed}, our proposed methods minimize upper bounds of PGO by solving independent optimization subproblems and are guaranteed to converge to first-order critical points. Furthermore, since our proposed methods rely on proximal operators of PGO, Nesterov's method \cite{nesterov1983method,nesterov2013introductory} can be used for acceleration, and the acceleration induces limited extra computation, and more importantly, no compromise of theoretical guarantees. In addition, we also propose accelerated majorization minimization methods for the distributed chordal initialization that have a quadratic convergence to compute an initial guess for distributed PGO.

The rest of this paper is organized as follows: \cref{section::notation} introduces notation used in this paper. \cref{section::problem} formulates the problem of distributed PGO. \cref{section::major} presents an upper bound that majorizes distributed PGO. \cref{section::mm} and \cref{section::amm} present and accelerate majorization minimization methods for distributed PGO, respectively. \cref{section::chordal} presents accelerated majorization minimization methods for the distributed chordal initialization. \cref{section::results} implements our proposed methods for distributed PGO in multi-robot SLAM and makes comparisons with existing state-of-the-art method \cite{choudhary2017distributed}. \cref{section::conclusion}  concludes the paper.

\section{Notation}\label{section::notation}
$\R$ denotes real numbers; $\R^{m\times n}$ and $\R^n$ denote $m\times n$ matrices and $n\times 1$ vectors, respectively; and $SO(d)$ and $SE(d)$ denote special orthogonal groups and special Euclidean groups, respectively. For a matrix $X\in \R^{m\times n}$, the notation $[X]_{ij}$ denotes its $(i,\,j)$-th entry or $(i,\,j)$-th block. The notation $\|\cdot\|$ denotes the Frobenius norm of matrices and vectors. For symmetric matrices $Y,\, Z\in \R^{n\times n}$, $Y\succeq Z$ (or $Z\preceq Y$) and $Y\succ Z$ (or $Z\prec Y$) indicate that $Y-Z$ is positive (or negative) semidefinite and definite, respectively. If $F:\R^{m\times n}\rightarrow\R $ is a function, $\mathcal{M}\subset \R^{m\times n}$ is a Riemannian manifold and $X\in \mathcal{M}$, the notation $\nabla F(X)$ and $\mathrm{grad}\, F(X)$ denote the Euclidean and Riemannian gradients, respectively. 

\section{Problem Formulation}\label{section::problem}
Distributed PGO considers the problem of estimating unknown poses $g_{1}^\alpha$, $g_{2}^\alpha$, $\cdots$, $g_{n_\alpha}^\alpha\in SE(d)$ of each node $\alpha\in \AA\triangleq\{1,\,2,\,\cdots,\, A\}$, in which $g_{(\cdot)}^\alpha=(t_{(\cdot)}^\alpha,\,R_{(\cdot)}^\alpha)$ with $t_{(\cdot)}^\alpha\in \R^d$ and $R_{(\cdot)}^\alpha\in SO(d)$ and $n_\alpha$ is the number of poses in node $\alpha$, given intra-node noisy measurements $\tilde{g}_{ij}^{\alpha\alpha}\in SE(d)$ of  
$$g_{ij}^{\alpha\alpha}\triangleq \left({g_i^\alpha}\right)^{-1} g_j^\alpha\in SE(d)$$
within a single node $\alpha$ and inter-node noisy measurements $\tilde{g}_{ij}^{\alpha\beta}\in SE(d)$ of  
$$g_{ij}^{\alpha\beta}\triangleq \left(g_i^\alpha\right)^{-1} g_j^\beta\in SE(d)$$
between different nodes $\alpha\neq\beta$. 

For notational simplicity, we rewrite poses $g_{1}^\alpha$, $g_{2}^\alpha$, $\cdots$, $g_{n_\alpha}^\alpha\in SE(d)$ of node $\alpha$ as 
\begin{equation}
X^\alpha\triangleq \begin{bmatrix}
t^\alpha & R^\alpha
\end{bmatrix}\in \XX^\alpha \subset \R^{d\times (d+1)n_\alpha}
\end{equation}
in which $\vphantom{\Big\{}t^\alpha\triangleq\begin{bmatrix}
t_1^\alpha & \cdots & t_{n_\alpha}^\alpha
\end{bmatrix}\in \R^{d\times n_\alpha}$, $R^\alpha\triangleq\begin{bmatrix}
R_1^\alpha & \cdots & R_{n_\alpha}^\alpha
\end{bmatrix}\in SO(d)^{n_\alpha}\subset \R^{d\times dn_\alpha}\vphantom{\Big\{}$, and
\begin{equation}
\nonumber
\XX^\alpha\triangleq\R^{d\times n_\alpha}\times SO(d)^{n_\alpha}.
\end{equation}
In addition, we define
$
\XX \triangleq  \XX^1\times\cdots\times \XX^A \subset \R^{d\times (d+1)n}
$
in which $n=\sum\limits_{\alpha\in \AA}n_\alpha$, and $\aEE^{\alpha\beta}$ such that  $(i,\,j)\in\aEE^{\alpha\beta}$ if and only if there exists a noisy measurement $\tilde{g}_{ij}^{\alpha\beta}\in SE(d)$, and $\NN_-^\alpha$ (respectively, $\NN_+^\alpha$) such that a node $\beta\in \NN_-^\alpha$ (respectively, $\beta\in\NN_+^\alpha$) if and only if $\aEE^{\alpha\beta}\neq\emptyset$ (respectively, $\aEE^{\beta\alpha}\neq \emptyset$) and $\beta\neq\alpha$.

Following \cite{rosen2016se}, it is possible to formulate distributed PGO as maximum likelihood estimation
\begin{equation}\label{eq::pgo}
\min_{X\in \XX} F(X).
\end{equation}
in which 
$X\triangleq\begin{bmatrix}
X^1 & \cdots & X^{A}
\end{bmatrix}\in \XX.$
In \cref{eq::pgo}, the objective function $F(X)$ is defined to be
\begin{multline}\label{eq::obj}
F(X)\triangleq \sum_{\alpha\in\AA}\sum_{(i,j)\in \aEE^{\alpha\alpha}}\frac{1}{2}\Big[\kappa_{ij}^{\alpha\alpha}\|R_i^\alpha \nR_{ij}^{\alpha\alpha} -R_j^\alpha\|^2 +\\ \tau_{ij}^{\alpha\alpha}\|R_i^\alpha \nt_{ij}^{\alpha\alpha}+t_i^\alpha - t_j^\alpha\|^2\Big]+\\
\sum_{\substack{\alpha,\beta\in\AA,\\
\alpha\neq \beta}}\sum_{(i,j)\in \aEE^{\alpha\beta}}\frac{1}{2}\Big[\kappa_{ij}^{\alpha\beta}\|R_i^\alpha \nR_{ij}^{\alpha\beta} -R_j^\beta\|^2 +\\ 
\tau_{ij}^{\alpha\beta}\|R_i^\alpha \nt_{ij}^{\alpha\beta}+t_i^\alpha - t_j^\beta\|^2\Big],
\end{multline}
in which $\kappa_{ij}^{\alpha\alpha}$, $\kappa_{ij}^{\alpha\beta}$, $\tau_{ij}^{\alpha\alpha}$ and $\tau_{ij}^{\alpha\beta}$ are weights that are related with measurement noise \cite{rosen2016se}. Furthermore, it is straightforward to show that there exists a positive-semidefinite data matrix $\nM\in \R^{(d+1)n\times(d+1)n}$ such that \cref{eq::obj} is equivalent to \cite{rosen2016se}
\begin{equation}\label{eq::objM}
F(X)\triangleq \frac{1}{2}\trace(X\nM X^\transpose).
\end{equation}

In the following sections, we will present majorization minimization methods to solve distributed PGO of \cref{eq::pgo,eq::objM}, which is the major contribution of this paper.
\section{The Majorization of Distributed PGO}\label{section::major}
In this section, following a similar procedure to \cite{fan2019proximal}, we will propose a function $E(X|X^{(k)})$ that is a proximal operator majorizing $F(X)$ in \cref{eq::obj,eq::objM}, and such a proximal operator is critical to our proposed majorization minimization methods for distributed PGO.

For any matrices $B,\,C\in \R^{m\times n}$, it is known that 
\begin{equation}\label{eq::inequality}
\frac{1}{2}\|B-C\|^2 \leq \|B-P\|^2 + \|C-P\|^2
\end{equation}
for any $P\in \R^{m\times n}$, in which ``$=$'' holds if
$$P=\frac{1}{2}B+\frac{1}{2}C. $$
If we assume that $X^{(k)}=\begin{bmatrix}
X^{1(k)} & \cdots & X^{A(k)}
\end{bmatrix}\in \XX$ with $X^{\alpha(k)}\in \XX^\alpha$ is the current iterate of \cref{eq::pgo}, then implementing \cref{eq::inequality} on each inter-robot measurement cost
$$\frac{1}{2}\kappa_{ij}^{\alpha\beta}\|R_i^\alpha \nR_{ij}^{\alpha\beta} -R_j^\beta\|^2 +\\ 
\frac{1}{2}\tau_{ij}^{\alpha\beta}\|R_i^\alpha \nt_{ij}^{\alpha\beta}+t_i^\alpha - t_j^\beta\|^2$$
between node $\alpha$ and $\beta$ ($\alpha\neq\beta$) in \cref{eq::obj}, we obtain an upper bound of \cref{eq::obj} as
\begin{multline}\label{eq::E}
E(X|X^{(k)})\triangleq \sum_{\alpha\in \AA}\sum_{(i,j)\in \aEE^{\alpha\alpha}}\frac{1}{2}\Big[\kappa_{ij}^{\alpha\alpha}\|R_i^\alpha \nR_{ij}^{\alpha\alpha} -R_j^\alpha\|^2 +\\ \tau_{ij}^{\alpha\alpha}\|R_i^\alpha \nt_{ij}^{\alpha\alpha}+t_i^\alpha - t_j^\alpha\|^2\Big]+\\
\sum_{\substack{\alpha,\beta\in \AA,\\\alpha\neq \beta}}\sum_{(i,j)\in \aEE^{\alpha\beta}}\Big[\kappa_{ij}^{\alpha\beta}\|R_i^\alpha \nR_{ij}^{\alpha\beta} -P_{ij}^{\alpha\beta(k)}\|^2 +\\ 
\tau_{ij}^{\alpha\beta}\|R_i^\alpha \nt_{ij}^{\alpha\beta}+t_i^\alpha - p_{ij}^{\alpha\beta(k)}\|^2+\\
\kappa_{ij}^{\alpha\beta}\|R_j^\beta -P_{ij}^{\alpha\beta(k)}\|^2+
\tau_{ij}^{\alpha\beta}\|t_j^\beta - p_{ij}^{\alpha\beta(k)}\|^2\Big],
\end{multline}   
in which
\begin{equation}\label{eq::P}
P^{\alpha\beta(k)}_{ij}=\frac{1}{2}R_i^{\alpha(k)}\nR_{ij}^{\alpha\beta}+\frac{1}{2}R_j^{\beta(k)}
\end{equation}
and
\begin{equation}
p^{\alpha\beta(k)}_{ij}=\frac{1}{2}R_i^{\alpha(k)}\nt_{ij}^{\alpha\beta}+\frac{1}{2}t_j^{\beta(k)}.
\end{equation}
As a matter of fact, we might decompose $E(X|X^{(k)})$ into
\begin{equation}\label{eq::Easum}
E(X|X^{(k)}) = \sum_{\alpha\in \AA} E^\alpha(X^\alpha|X^{(k)}),
\end{equation}
in which each $E^\alpha(X^\alpha|X^{(k)})$ only depends on $X^{\alpha}$ in a single node $\alpha$
\begin{multline}\label{eq::Ea}
E^\alpha(X^\alpha|X^{(k)})\triangleq \sum_{(i,j)\in \aEE^{\alpha\alpha}}\frac{1}{2}\Big[\kappa_{ij}^{\alpha\alpha}\|R_i^\alpha \nR_{ij}^{\alpha\alpha} -R_j^\alpha\|^2 +\\ \tau_{ij}^{\alpha\alpha}\|R_i^\alpha \nt_{ij}^{\alpha\alpha}+t_i^\alpha - t_j^\alpha\|^2\Big]+\\
\sum_{\substack{\beta\in \NN_-^\alpha}}\sum_{(i,j)\in \aEE^{\alpha\beta}}\Big[\kappa_{ij}^{\alpha\beta}\|R_i^\alpha \nR_{ij}^{\alpha\beta} -P_{ij}^{\alpha\beta(k)}\|^2 +\\ 
\tau_{ij}^{\alpha\beta}\|R_i^\alpha \nt_{ij}^{\alpha\beta}+t_i^\alpha - p_{ij}^{\alpha\beta(k)}\|^2\Big]+\\
+\sum_{\substack{\beta\in \NN_+^\alpha}}\sum_{(j,i)\in \aEE^{\beta\alpha}}\Big[\kappa_{ji}^{\beta\alpha}\|R_i^\alpha -P_{ji}^{\beta\alpha(k)}\|^2+
\tau_{ji}^{\beta\alpha}\|t_i^\alpha - p_{ji}^{\beta\alpha(k)}\|^2\Big].
\end{multline}
Moreover, $E(X|X^{(k)})$ is a proximal operator of $F(X)$ as the following proposition states.

\begin{prop}\label{prop::G}
For all $\alpha\in\AA$, there exists constant positive-semidefinite matrices $\nH^\alpha\in \R^{(d+1)n_\alpha\times (d+1)n_\alpha}$ such that 
$E(X|X^{(k)})$ is equivalent to 
\begin{multline}\label{eq::EM}
E(X|X^{(k)}) \triangleq \frac{1}{2}\innprod{\nH(X-X^{(k)})}{X-X^{(k)}}+\\
\innprod{\nabla F(X^{(k)})}{{X-X^{(k)}}}+ F(X^{(k)}),
\end{multline}
in which $\nH \in \R^{(d+1)n\times (d+1)n}$ is a block diagonal matrix
\begin{equation}
\nonumber
\nH\triangleq\mathrm{diag}\big\{\nH^1,\,\cdots,\,\nH^A\big\}\in \R^{(d+1)n\times (d+1)n},
\end{equation}
and $\nabla F(X^{(k)}) \triangleq X^{(k)}\nM\in \R^{d\times(d+1)n}$ is the Euclidean gradient of $F(X)$ at $X^{(k)}\in\XX$.
Furthermore, $\nH\succeq \nM$ and $E(X|X^{(k)})\geq F(X^{(k)})$ in which ``$=$'' holds if  $X=X^{(k)}$.
\end{prop}
\begin{proof}
See \cite[Appendix A]{fan2020mm_full}.
\end{proof}

In the next section, we will present majorization minimization methods for distributed PGO using $E(X|X^{(k)})$ in \cref{eq::E,eq::EM}.

\section{The Majorization Minimization Method for Distributed PGO}\label{section::mm}
From \cref{prop::G}, it is known $\nH\succeq \nM$, and then it can be shown that if $\xi \in\R$ and $\xi \geq 0$, there exists a block diagonal matrix $\nGamma\triangleq\nH + \xi\cdot\I\in \R^{(d+1)n\times (d+1)n}$ such that
\begin{equation}\label{eq::Gamma}
\nGamma\triangleq \diag\big\{\nGamma^1,\,\cdots,\, \nGamma^A\big\} \succeq \nM ,
\end{equation}
in which
\begin{equation}\label{eq::Gammaxi}
\nGamma^\alpha\triangleq \nH^\alpha +\xi\cdot\I^\alpha\in \R^{(d+1)n_\alpha \times (d+1)n_\alpha}
\end{equation} 
and $\I^\alpha\in \R^{(d+1)n_\alpha \times (d+1)n_\alpha}$ is the identity matrix. Then, we obtain
\begin{multline}\label{eq::G}
G(X|X^{(k)}) \triangleq \frac{1}{2}\innprod{\nGamma(X-X^{(k)})}{X-X^{(k)}}+\\
\innprod{\nabla F(X^{(k)})}{{X-X^{(k)}}}+ F(X^{(k)})
\end{multline}
such that $G(X|X^{(k)})\geq F(X)$ in which ``$=$'' holds if  $X=X^{(k)}$. More importantly, if there exists $X^{(k+1)}$ with 
\begin{equation}\label{eq::major}
G(X^{(k+1)}|X^{(k)})\leq G(X^{(k)}|X^{(k)}),
\end{equation}
then it can be concluded that 
$$F(X^{(k+1)})\leq G(X^{(k+1)}|X^{(k)})\leq G(X^{(k)}|X^{(k)}) =F(X^{(k)}).$$
In general, such a $X^{(k+1)}$ satisfying \cref{eq::major} can be found by solving 
\begin{equation}\label{eq::optG}
\min_{X\in \XX} G(X|X^{(k)}).
\end{equation}
Furthermore, since $\nGamma$ is a block diagonal matrix and
\begin{equation}\label{eq::nabF}
\nabla F(X^{(k)})\triangleq\begin{bmatrix}
\nabla_1 F(X^{(k)}) & \cdots & \nabla_A F(X^{(k)})
\end{bmatrix},
\end{equation}
in which $\nabla_\alpha F(X)\in \R^{d\times (d+1)n_\alpha}$ is the Euclidean gradient of $F(X)$ with respect to $X^{\alpha}\in \XX^\alpha$,  we might decompose $G(X|X^{(k)})$ into
\begin{equation}\label{eq::Gsub}
G(X|X^{(k)})= \sum_{X^\alpha\in \XX^\alpha} G^\alpha(X^{\alpha}|X^{(k)}),
\end{equation}
such that each $G^\alpha(X^{\alpha}|X^{(k)})$ is only related with $X^{\alpha}$ and $X^{\alpha(k)}\in\XX^\alpha$
\begin{multline}\label{eq::Ga}
G^\alpha(X^{\alpha}|X^{(k)})\triangleq \frac{1}{2}\innprod{\nGamma^\alpha(X^\alpha-X^{\alpha(k)})}{X^\alpha-X^{\alpha(k)}}+\\
\innprod{\nabla_\alpha F(X^{(k)})}{{X^\alpha-X^{\alpha(k)}}}+\overline{G}^{\alpha(k)},
\end{multline}
in which
\begin{multline}\label{eq::Fa}
\overline{G}^{\alpha(k)}\triangleq \sum_{(i,j)\in \aEE^{\alpha\alpha}}\frac{1}{2}\Big[\kappa_{ij}^{\alpha\alpha}\|R_i^{\alpha(k)} \nR_{ij}^{\alpha\alpha} -R_j^{\alpha(k)}\|^2 +\\ \tau_{ij}^{\alpha\alpha}\|R_i^{\alpha(k)} \nt_{ij}^{\alpha\alpha}+t_i^{\alpha(k)} - t_j^{\alpha(k)}\|^2\Big]+\\
\sum_{\substack{\beta\in \AA,\\\alpha\neq \beta}}\sum_{(i,j)\in \aEE^{\alpha\beta}}\frac{1}{4}\Big[\kappa_{ij}^{\alpha\beta}\|R_i^{\alpha(k)} \nR_{ij}^{\alpha\beta} -R_{j}^{\beta(k)}\|^2 +\\
\tau_{ij}^{\alpha\beta}\|R_i^{\alpha(k)} \nt_{ij}^{\alpha\beta}+t_i^{\alpha(k)} - t_j^{\beta(k)}\|^2\Big]+\\
\sum_{\substack{\beta\in \AA,\\\alpha\neq \beta}}\sum_{(i,j)\in \aEE^{\beta\alpha}}\frac{1}{4}\Big[\kappa_{ij}^{\beta\alpha}\|R_i^{\beta(k)} \nR_{ij}^{\beta\alpha} -R_{j}^{\alpha(k)}\|^2 +\\
\tau_{ij}^{\beta\alpha}\|R_i^{\beta(k)} \nt_{ij}^{\beta\alpha}+t_i^{\beta(k)} - t_j^{\alpha(k)}\|^2\Big].
\end{multline}
As a result, it is straightforward to show that \cref{eq::optG} is equivalent to solving $A$ independent optimization subproblems of smaller size
\begin{equation}\label{eq::optGsub}
\min_{X^\alpha\in \XX^\alpha} G^\alpha(X^\alpha|X^{(k)}),\quad\quad \forall\alpha=1,\,\cdots,\,A.
\end{equation}
If $X^{\alpha(k+1)}$ is computed from \cref{eq::optGsub}, it can be concluded that $F(X^{(0)})$, $F(X^{(1)})$, $\cdots$ are non-increasing as long as $G^\alpha(X^{\alpha(k+1)}|X^{(k)})\leq G^\alpha(X^{\alpha(k)}|X^{(k)})$ for each node $\alpha$.

\begin{algorithm}[t]
	\caption{The $\mm$ Method}
	\label{algorithm::mm}
	\begin{algorithmic}[1]
		\State\textbf{Input}: An initial iterate $X^{(0)}\in \XX$.
		\State\textbf{Output}: A sequence of iterates $\{X^{(k)}\}$.\vspace{0.2em} 
		\vspace{0.25em}
		\State $\nabla F(X^{(0)})\leftarrow X^{(0)}\nM$
		\vspace{0.25em}
		\For{$k=0,\,1,\,2,\,\cdots$}
		\vspace{0.2em}
		\For{$\alpha= 1,\,\cdots,\, A$}  
		\vspace{0.2em}
		\State $X^{\alpha(k+1)}\leftarrow\arg\min\limits_{X^\alpha\in\XX^\alpha }G^\alpha(X^\alpha|X^{(k)})$
		\EndFor
		\vspace{0.2em}
		\State $X^{(k+1)}\leftarrow\begin{bmatrix}
		X^{1(k+1)} & \cdots & X^{A(k+1)}
		\end{bmatrix}$
		\State $\nabla F(X^{(k+1)})\leftarrow X^{(k+1)}\nM  $\label{line::grad_F}
		\EndFor
	\end{algorithmic}
\end{algorithm}

Following \cref{eq::optG,eq::optGsub}, we obtain the $\mm$ method for distributed PGO (\cref{algorithm::mm}). The $\mm$ method can be classified as the majorization minimization method \cite{hunter2004tutorial} that is widely used in machine learning, signal processing, applied mathematics, etc. The $\mm$ method can be distributed without introducing any extra computational workloads as long as each node $\alpha$ can communicate with its neighbor node $\beta\in \NN_-^\alpha\cup \NN_+^\alpha$, i.e., each pose $g_i^\alpha$ in node $\alpha$ only needs pose $g_j^{\beta}$ in node $\beta$ for which either $(i,\,j)\in\aEE^{\alpha\beta}$ or $(j,\,i)\in\aEE^{\beta\alpha}$ to evaluate $\nabla F(X^{(k)})$ in line~\ref{line::grad_F} of \cref{algorithm::mm}. Even though $\overline{G}^{\alpha(k)}$ is included in $G^\alpha(X^\alpha|X^{(k)})$, it does not have to be explicitly evaluated when we minimize $G^\alpha(X^\alpha|X^{(k)})$. Furthermore, as is shown below, if \cref{assumption::local} holds, we obtain \cref{prop::mm} that the $\mm$ method is guaranteed to converge to first-order critical points of distributed PGO under mild conditions.

\begin{assumption}\label{assumption::local}
	For $X^{\alpha(k+1)}\leftarrow\arg\min\limits_{X^\alpha\in\XX^\alpha }G^\alpha(X^\alpha|X^{(k)})$ in the $\mm$ method, it is assumed that 
	$$G^\alpha(X^{\alpha(k+1)}|X^{(k)})\leq G^\alpha(X^{\alpha(k)}|X^{(k)})$$
	and 
	$$\mathrm{grad}\,G^\alpha(X^{\alpha(k+1)}|X^{(k)}) =\mathbf{0}$$
	for all $\alpha=1,\,\cdots,\, A$.
\end{assumption}

\begin{prop}\label{prop::mm}
If \cref{assumption::local} holds, then for a sequence of iterates $\{X^{(k)}\}$ generated by \cref{algorithm::mm}, we obtain
\begin{enumerate}[(1)]
\item  $F(X^{(k)})$ is non-increasing;\label{prop::mm1}
\item  $F(X^{(k)})\rightarrow F^\infty$ as $k\rightarrow\infty$;\label{prop::mm2}
\item  $\|X^{(k+1)}-X^{(k)}\|\rightarrow 0$ as $k\rightarrow \infty$ if $\nGamma\succ\nM$;\label{prop::mm3}
\item  $\mathrm{grad}\, F(X^{(k)})\rightarrow \0$ as $k\rightarrow \infty$ if $\nGamma\succ\nM$;\label{prop::mm4}
\item  $\|X^{(k+1)}-X^{(k)}\|\rightarrow 0$ as $k\rightarrow \infty$ if $\xi >0$;\label{prop::mm5}
\item  $\mathrm{grad}\, F(X^{(k)})\rightarrow \0$ as $k\rightarrow \infty$ if $\xi > 0$.\label{prop::mm6}
\end{enumerate}
\end{prop}
\begin{proof}
See \cite[Appendix B]{fan2020mm_full}.
\end{proof}

It should be noted that the $\mm$ method differs from \cite{choudhary2017distributed}. The distributed PGO method in \cite{choudhary2017distributed} relies on iterative distributed linear system solvers to evaluate the Gauss-Newton direction and then update the estimate using a single Gauss-Newton step, whereas the $\mm$ method in our paper minimizes an upper bound of PGO that is guaranteed to improve the current estimate and no Gauss-Newton directions are evaluated. Furthermore, to our knowledge, the distributed method in \cite{choudhary2017distributed} is not non-increasing and has no convergence guarantees.

\section{The Accelerated Majorization Minimization Method for Distributed PGO}\label{section::amm}
\begin{algorithm}[t]
	\caption{The $\amm$ Method}
	\label{algorithm::amm}
	\begin{algorithmic}[1]
		\State\textbf{Input}: An initial iterate $X^{(0)}\in \XX$.
		\State\textbf{Output}: A sequence of iterates $\{X^{(k)}\}$.\vspace{0.2em} 	
		\vspace{0.25em}
		\State $\nabla F(X^{(0)})\leftarrow X^{(0)}\nM$\label{line::alg2::grad_F0}
		\vspace{0.1em}
		\For{$\alpha=1,\,\cdots,\, A$}
		\State evaluate $\overline{G}^{\alpha(0)}$ using \cref{eq::Fa}\label{line::alg2::F0}
		\vspace{0.2em}
		\State $s^{\alpha(0)}\leftarrow 1$ 
		\vspace{0.2em}
		\State $X^{\alpha(-1)}\leftarrow X^{\alpha(0)}$,\; $\nabla F(X^{\alpha(-1)})\leftarrow \nabla F(X^{\alpha(0)})$
		\EndFor
		\For{$k=0,\,1,\,2,\,\cdots$}
		\vspace{0.3em}
		\For{$\alpha=1,\,\cdots,\, A$}  
		\State $s^{\alpha(k+1)}\leftarrow\tfrac{\sqrt{4{s^{\alpha(k)}}^2+1}+1}{2}$,\; $\gamma^{\alpha(k)}\leftarrow \tfrac{s^{\alpha(k)}-1}{s^{\alpha(k+1)}}$\label{line::s}
		\vspace{0.3em}
		\State $Y^{\alpha(k)}\leftarrow X^{\alpha(k)}+\gamma^{\alpha(k)}\cdot\left(X^{\alpha(k)}-X^{\alpha(k-1)}\right)$\label{line::nesterov1}
		\vspace{0.3em}
		\State\label{line::nesterov2}
		\vspace{-2em}
		\begin{multline}
		\nonumber
		\quad\quad\nabla F(Y^{\alpha(k)})\leftarrow \nabla F(X^{\alpha(k)})+\gamma^{\alpha(k)}\cdot\\
		\left(\nabla F(X^{\alpha(k)})-\nabla F(X^{\alpha(k-1)})\right)
		\end{multline} 
		\vspace{-1.6em}
		\vspace{0.3em}
		\State$Z^{\alpha(k+1)}\leftarrow\arg\min\limits_{Z^\alpha\in\XX^\alpha }G^\alpha(Z^\alpha|Y^{\alpha(k)})$
		\vspace{0.4em}
		\If{$G^\alpha(Z^{\alpha(k+1)}|X^{(k)}) > \overline{G}^{\alpha(k)}$}\label{line::restart1}
		\vspace{0.3em}
		\State$X^{\alpha(k+1)}\leftarrow\arg\min\limits_{X^\alpha\in\XX^\alpha }G^\alpha(X^\alpha|X^{(k)})$
		\vspace{0.3em}
		\State $s^{\alpha(k+1)}\leftarrow \max\{\tfrac{1}{2}s^{\alpha(k+1)},\,1\}$
		\Else
		\State $X^{\alpha(k+1)}\leftarrow Z^{\alpha(k+1)}$
		\EndIf\label{line::restart2}
		\EndFor
		\vspace{0.15em}
		\State $X^{(k+1)}\leftarrow\begin{bmatrix}
		X^{1(k+1)} & \cdots & X^{A(k+1)}
		\end{bmatrix}$
		\vspace{0.1em}
		\State $\nabla F(X^{(k+1)})\leftarrow X^{(k+1)}\nM$\label{line::alg2::grad_Fk}
		\vspace{0.1em}
		\For{$\alpha=1,\,\cdots,\, A$}
		\vspace{0.25em}
		\State evaluate $\overline{G}^{\alpha(k+1)}$ using \cref{eq::Fa}\label{line::alg2::Fk}
		\vspace{0.1em}
		\EndFor
		\EndFor
	\end{algorithmic}
\end{algorithm}

In the last thirty years, a number of accelerated first-order optimization methods have been proposed \cite{nesterov1983method,nesterov2013introductory}. Even though most of these accelerated methods were originally developed for convex optimization, it has been recently found that they empirically have a good performance for nonconvex optimization as well \cite{ghadimi2016accelerated,jin2018accelerated}. 

From \cref{eq::G}, it can be seen that $G(X|X^{(k)})$ is a proximal operator of $F(X)$, which suggests that the $\mm$ method is a proximal method, and most importantly, it is possible to exploit existing accelerated schemes for proximal methods \cite{nesterov1983method,nesterov2013introductory}. 

Similar to \cite{fan2019proximal},  we might extend the $\mm$ method to obtain the accelerated majorization minimization method for distributed PGO using Nesterov's method \cite{nesterov1983method,nesterov2013introductory}. The resulting algorithm is referred as the $\amm$ method (\cref{algorithm::amm}). For the $\amm$ method, each node $\alpha$ only needs pose estimates $g_j^{\beta(k)}$ of its neighbor node $\beta$ to evaluate $\nabla F(X^{(k)})$ and $\overline{G}^{\alpha(k)}$ in lines~\ref{line::alg2::grad_F0}, \ref{line::alg2::F0}, \ref{line::alg2::grad_Fk} and \ref{line::alg2::Fk} of \cref{algorithm::amm}. The $\amm$ method is equivalent to the $\mm$ method  when $s^{\alpha(k)}=1$, and is more governed by Nesterov's momentum as $s^{\alpha(k)}$ increases. From \cref{algorithm::amm}, the $\amm$ method  introduces Nesterov's momentum in lines~\ref{line::s} to \ref{line::nesterov2} for acceleration, and adopts a restart in lines~\ref{line::restart1} to \ref{line::restart2} to guarantee the convergence and improve the overall performance. In \cref{algorithm::amm}, there is no need to evaluate the objective of PGO, i.e., \cref{eq::obj,eq::objM}, which differs from the algorithm in \cite{fan2019proximal}, and thus, it is well-suited for distributed PGO. As \cref{prop::amm} states, the $\amm$ method has $F(X^{(k)})$ non-increasing and is guaranteed to converge to first-order critical points as long as \cref{assumption::local,assump::localZ} hold and $\xi > 0$. 

\begin{assumption}\label{assump::localZ}
For $Z^{\alpha(k+1)}\leftarrow\arg\min\limits_{Z^\alpha\in\XX^\alpha }G^\alpha(Z^\alpha|Y^{\alpha(k)})$ in the $\amm$ method, it is assumed that 
$$\mathrm{grad}\,G^\alpha(Z^{\alpha(k+1)}|Y^{\alpha(k)}) =\mathbf{0}$$
for all $\alpha=1,\,\cdots,\, A$.
\end{assumption}

\begin{prop}\label{prop::amm}
	If \cref{assumption::local,assump::localZ} hold, then for a sequence of iterates $\{X^{(k)}\}$ generated by \cref{algorithm::amm}, we obtain
	\begin{enumerate}[(1)]
		\item  $F(X^{(k)})$ is non-increasing;\label{prop::amm0}
		\item  $F(X^{(k)})\rightarrow F^\infty$ as $k\rightarrow\infty$;\label{prop::amm1}
		\item  $\|X^{(k+1)}-X^{(k)}\|\rightarrow 0$ as $k\rightarrow \infty$ if $\nGamma\succ\nM$;\label{prop::amm2}
		\item  $\mathrm{grad}\, F(X^{(k)})\rightarrow \0$ as $k\rightarrow \infty$ if $\nGamma\succ\nM$;\label{prop::amm3}
		\item  $\|X^{(k+1)}-X^{(k)}\|\rightarrow 0$ as $k\rightarrow \infty$ if $\xi >0$;\label{prop::amm4}
		\item  $\mathrm{grad}\, F(X^{(k)})\rightarrow \0$ as $k\rightarrow \infty$ if $\xi > 0$.\label{prop::amm5}
	\end{enumerate}
\end{prop}
\begin{proof}
See \cite[Appendix C]{fan2020mm_full}.
\end{proof}

\section{The Majorization Minimization Method for The Distributed Chordal Initialization}\label{section::chordal}
In PGO, the chordal initialization is one of the most popular initialization techniques \cite{carlone2015initialization}, however, the distributed chordal initialization remains challenging \cite{choudhary2017distributed}. In this section, we will present majorization minimization methods for the distributed chordal initialization that have a quadratic convergence. 

The chordal initialization relaxes $SO(d)^n$ to $\R^{d\times dn}$ and solves the convex optimization problem
\vspace{-0.5em}
\begin{equation}\label{eq::chordal}
\min_{R\in \RR} F_R(R)
\vspace{-.5em}
\end{equation}
In \cref{eq::chordal}, $F_R(R)$, $R$ and $\RR$ are respectively defined to be
\vspace{-1.5em}  
\begin{multline}\label{eq::objR}
F_R(R)\triangleq \sum_{\alpha\in\AA}\sum_{(i,j)\in \aEE^{\alpha\alpha}}\frac{1}{2}\kappa_{ij}^{\alpha\alpha}\|R_i^\alpha \nR_{ij}^{\alpha\alpha} -R_j^\alpha\|^2 +\\ 
\sum_{\substack{\alpha,\beta\in\AA,\\\alpha\neq \beta}}\sum_{(i,j)\in \aEE^{\alpha\beta}}\frac{1}{2}\kappa_{ij}^{\alpha\beta}\|R_i^\alpha \nR_{ij}^{\alpha\beta} -R_j^\beta\|^2,
\vspace{-1em}
\end{multline}
and $R\triangleq\begin{bmatrix}
R^1 & \cdots & R^A
\end{bmatrix}\in \R^{d\times dn}$ in which $R^\alpha\triangleq\begin{bmatrix}
R_1^\alpha & \cdots & R_{n_\alpha}^\alpha
\end{bmatrix}\in \R^{d\times dn_\alpha}$, and $\mathcal{R}\triangleq \RR^1\times\cdots\times \RR^A $ in which $\RR^1\triangleq\{R^1\in \R^{d\times dn_1}| R_1^1=\I\in \R^{d\times d}\}$ and $\RR^\alpha\triangleq \R^{d\times dn_\alpha}$  if $\alpha\neq 1$. From \cref{eq::inequality}, if $R^{(k)}$ is the current estimate of $R$, we might obtain an upper bound $G_R(R|R^{(k)})$ of $F(R)$
\vspace{-1.5em}
\begin{multline}\label{eq::GR}
G_R(R|R^{(k)})\triangleq \sum_{\alpha\in \AA}\sum_{(i,j)\in \aEE^{\alpha\alpha}}\frac{1}{2}\kappa_{ij}^{\alpha\alpha}\|R_i^\alpha \nR_{ij}^{\alpha\alpha} -R_j^\alpha\|^2+\\
\sum_{\substack{\alpha,\beta\in \AA,\\\alpha\neq \beta}}\sum_{(i,j)\in \aEE^{\alpha\beta}}\Big[\kappa_{ij}^{\alpha\beta}\|R_i^\alpha \nR_{ij}^{\alpha\beta} -P_{ij}^{\alpha\beta(k)}\|^2 +\\ 
+\kappa_{ij}^{\alpha\beta}\|R_j^\beta -P_{ij}^{\alpha\beta(k)}\|^2\Big]+\frac{1}{2}\xi \|R-R^{(k)}\|^2,
\end{multline}  
in which $P_{ij}^{\alpha\beta(k)}\in \R^{d\times d}$ is defined as \cref{eq::P} and $\xi\geq0$.

\begin{algorithm}[t]
	\caption{The $\nag$ Method}
	\label{algorithm::nag}
	\begin{algorithmic}[1]
		\State\textbf{Input}: An initial iterate $R^{(0)}\in \RR$.\label{line::nag::input}
		\State\textbf{Output}: A sequence of iterates $\{R^{(k)}\}$.\vspace{0.2em} 	
		\vspace{0.25em}
		\State evaluate $\nabla F_R(R^{(0)})$
		\vspace{0.1em}
		\For{$\alpha=1,\,\cdots,\, A$}
		\vspace{0.2em}
		\State $G_R^{\alpha(0)}\leftarrow G_R^{\alpha}(X|X^{\alpha(0)})$
		\vspace{0.2em}
		\State $s^{\alpha(0)}\leftarrow 1$ 
		\vspace{0.2em}
		\State $R^{\alpha(-1)}\leftarrow R^{\alpha(0)}$,\; $\nabla F_R(R^{\alpha(-1)})\leftarrow \nabla F_R(R^{\alpha(0)})$
		\EndFor
		\For{$k=0,\,1,\,2,\,\cdots$}
		\vspace{0.3em}
		\For{$\alpha=1,\,\cdots,\, A$} 
		\vspace{0.2em} 
		\State $s^{\alpha(k+1)}\leftarrow\tfrac{\sqrt{4{s^{\alpha(k)}}^2+1}+1}{2}$,\; $\gamma^{\alpha(k)}\leftarrow \tfrac{s^{\alpha(k)}-1}{s^{\alpha(k+1)}}$\label{line::alg3::s}
		\vspace{0.3em}
		\State $Y^{\alpha(k)}\leftarrow R^{\alpha(k)}+\gamma^{\alpha(k)}\cdot\left(R^{\alpha(k)}-R^{\alpha(k-1)}\right)$
		\vspace{0.3em}
		\State
		\vspace{-2em}
		\begin{multline}
		\nonumber
		\quad\quad\nabla F_R(Y^{\alpha(k)})\leftarrow \nabla F_R(R^{\alpha(k)})+\gamma^{\alpha(k)}\cdot\\
		\left(\nabla F_R(R^{\alpha(k)})-\nabla F_R(R^{\alpha(k-1)})\right)
		\end{multline} 
		\vspace{-1.6em}
		\vspace{0.3em}
		\State $R^{\alpha(k+1)}\leftarrow\arg\min\limits_{Z^\alpha\in\RR^\alpha }G_R^\alpha(Z^\alpha|Y^{\alpha(k)})$
		\vspace{0.4em}
		\EndFor
		\vspace{0.15em}
		\State $R^{(k+1)}\leftarrow\begin{bmatrix}
		R^{1(k+1)} & \cdots & R^{A(k+1)}
		\end{bmatrix}$
		\vspace{0.3em}
		\State evaluate $\nabla F_R(R^{(k+1)})$
		\vspace{0.1em}
		\EndFor
	\end{algorithmic}
\end{algorithm}

\begin{figure*}[t]
	\centering
	\begin{tabular}{ccc}
		\hspace{-0.25em}\subfloat[][]{\includegraphics[trim =0mm 0mm 0mm 0mm,width=0.28\textwidth]{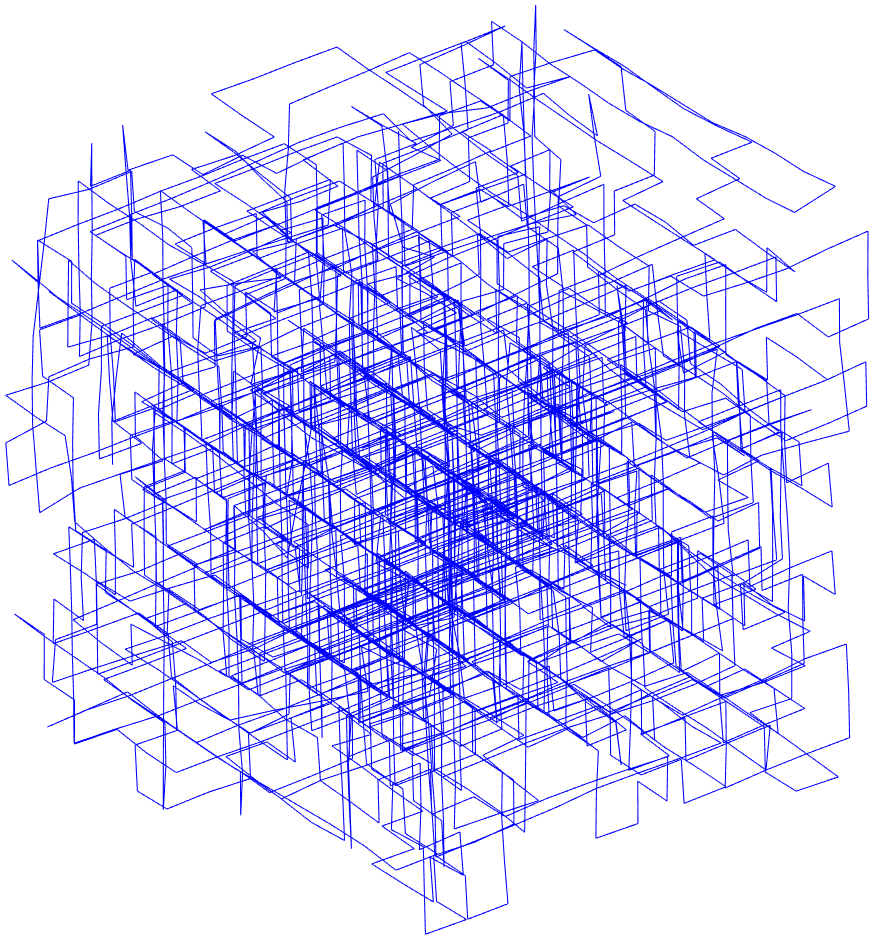}} &
		\hspace{-0.25em}\subfloat[][]{\includegraphics[trim =0mm 0mm 0mm 0mm,width=0.32\textwidth]{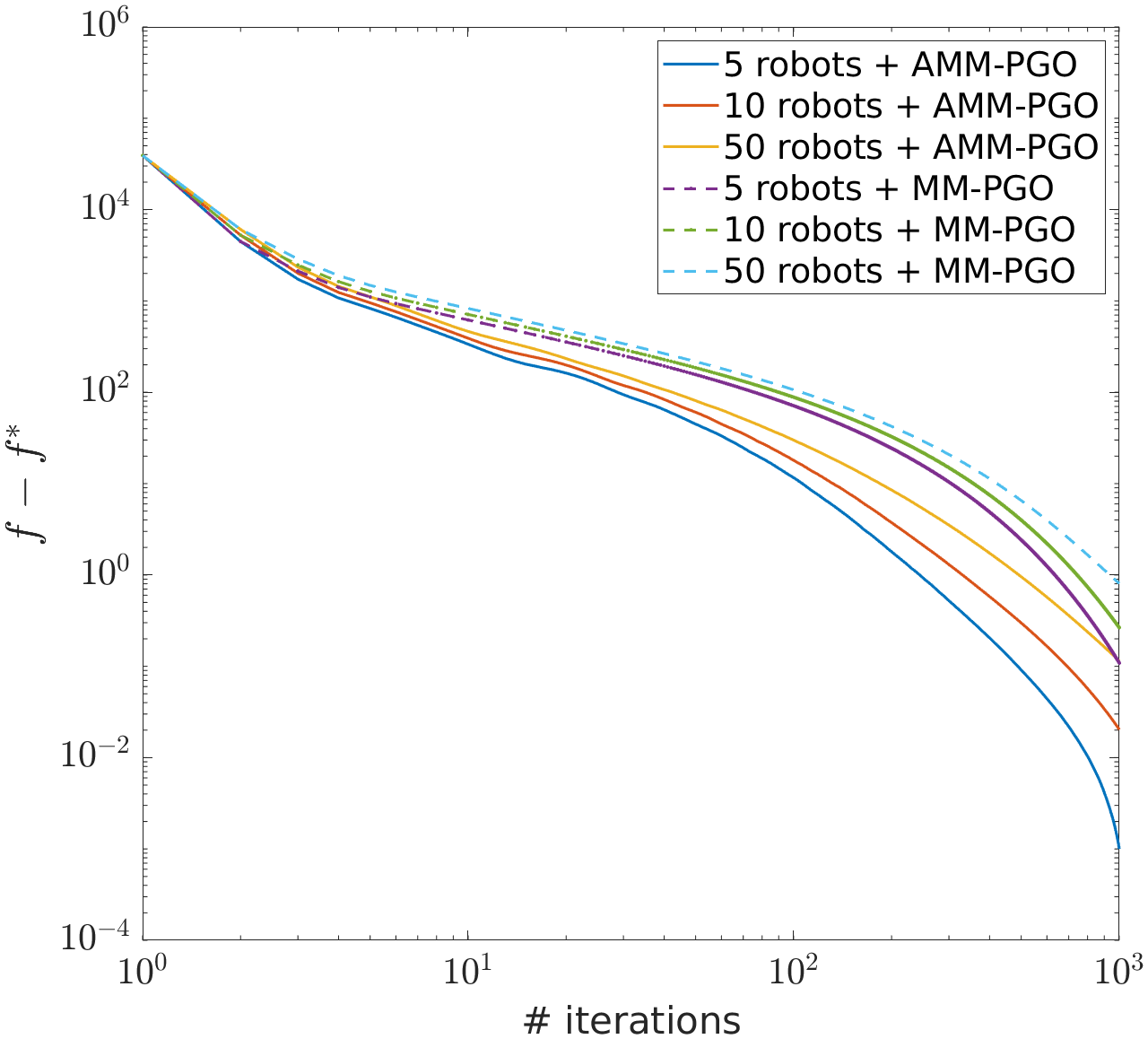}} &
		\hspace{-0.25em}\subfloat[][]{\includegraphics[trim =0mm 0mm 0mm 0mm,width=0.32\textwidth]{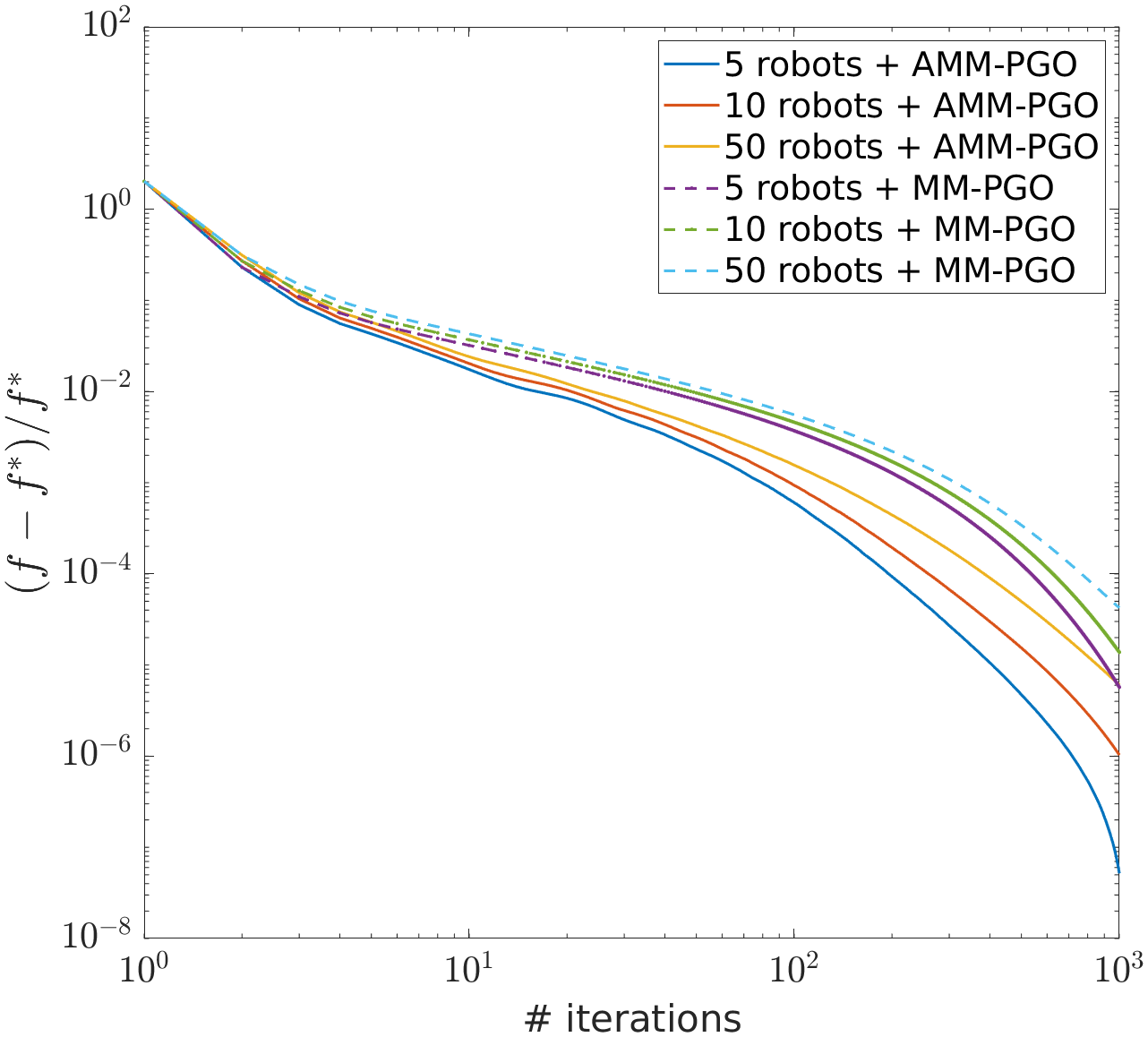}}  
	\end{tabular}
	\caption{The comparisons of the $\mm$ and $\amm$ methods on $20$ {\sf\small Cube} datasets with $5$, $10$ and $50$ robots in which the maximum number of iterations is $1000$. Each {\sf\small Cube} dataset has $12\times12\times12$ grids of side length of $1$ m, $3600$ poses, probability of loop closure of $0.1$, an translational RSME of $\sigma_t = 0.02$ m and an angular RSME of $\sigma_R=0.02\pi$ rad. The results are (a) an example of the {\sf\small Cube} dataset,  (b) suboptimality gap $f-f^*$ and (c) relative suboptimality gap $(f-f^*)/f^*$. In (b) and (c), $f$ is the objective attained by the $\mm$ and $\amm$ methods and $f^*$ is the globally optimal objective attained by SE-Sync \cite{rosen2016se}.}
		\label{fig::cube} 
		\vspace{-1em}
\end{figure*}

In a similar way to $G(X|X^{(k)})$ in \cref{eq::G}, it can be shown that $G_R(R|R^{(k)})$ is a proximal operator of $F_R(R)$ at $R^{(k)}$ and there exists $G^\alpha(R^{\alpha}|R^{\alpha(k)})$ such that $G_R(R|R^{(k)})=\sum_{i=1}^A G_R^\alpha(R^\alpha|R^{\alpha(k)})$ and
\begin{equation}
\nonumber
\min_{R\in \RR} G_R(R|R^{(k)})
\end{equation}
is equivalent to solving $A$ independent convex optimization subproblems
\begin{equation}
\nonumber
\min_{R^\alpha\in \RR^\alpha} G_R^\alpha(R^\alpha|R^{\alpha(k)}).
\end{equation}
From Nesterov's method \cite{nesterov1983method,nesterov2013introductory}, we obtain the $\nag$ method (\cref{algorithm::nag}), which is an accelerated majorization minimization method to solve the distributed chordal initialization of \cref{eq::chordal}. Furthermore, since the chordal initialization is a convex optimization problem, the $\nag$ method quadratically converges to the global optimum of \cref{eq::chordal} as follows.

\begin{prop}\label{prop::chordal}
The $\nag$ method has a convergence rate of $O(1/k^2)$ to the global optimum of the distributed choral initialization of \cref{eq::chordal}.
\end{prop}
\begin{proof}
See \cite[Appendix D]{fan2020mm_full}.
\end{proof}

The resulting solution to the chordal initialization might not satisfy the orthogonal constraints, and we need to project each $R_i^\alpha$ from $\R^{d\times d}$ to $SO(d)$ using the singular value decomposition \cite{umeyama1991least} to get the initial guess $R^{(0)}\in SO(d)^n$ of the rotation $R\in SO(d)^n$. 

It is possible to further obtain an initial guess $t^{(0)}\in t^{d\times n}$ of the translation $t\triangleq\begin{bmatrix}
t^1 & \cdots & t^A
\end{bmatrix}\in t^{d\times n}$ with $t^\alpha\triangleq\begin{bmatrix}
t_1^\alpha & \cdots & t_{n_\alpha}^\alpha
\end{bmatrix}\in \R^{d\times n_\alpha}$ by substituting $R^{(0)}$ into \cref{eq::obj} and solving the optimization problem
\vspace{-0.5em}
\begin{equation}\label{eq::chordal_t}
\min_{t\in \TT} F_t(t),
\end{equation}
\vspace{-0.25em}
in which
\vspace{-0.75em}  
\begin{multline}\label{eq::objt}
F_t(t)\triangleq \sum_{\alpha\in\AA}\sum_{(i,j)\in \aEE^{\alpha\alpha}}\frac{1}{2}\tau_{ij}^{\alpha\alpha}\|R_i^{\alpha(0)} \nt_{ij}^{\alpha\alpha}+t_i^\alpha - t_j^\alpha\|^2 +\\ 
\sum_{\substack{\alpha,\beta\in\AA,\\\alpha\neq \beta}}\sum_{(i,j)\in \aEE^{\alpha\beta}}\frac{1}{2}\tau_{ij}^{\alpha\beta}\|R_i^{\alpha (0)}\nt_{ij}^{\alpha\beta}+t_i^\alpha - t_j^\beta\|^2,
\end{multline}
and $\TT\triangleq \TT^1\times\cdots\times \TT^A\subset\R^{d\times A} $ with $\TT^1\triangleq\{t^1\in \R^{d\times n_1}| t_1^1=\0\in \R^{d}\}$ and $\TT^\alpha\triangleq \R^{d\times n_\alpha}$  if $\alpha\neq 1$. Following a similar procedure to \cref{eq::chordal}, \cref{eq::chordal_t} can be solved with the majorization minimization method, from which we obtain an initial guess  $t^{(0)}\in \R^{d\times n}$ of the translation $t\in \R^{d\times n}$.

\vspace{-0.15em}
\section{Numerical Experiments}\label{section::results}
In this section, we evaluate the performance of our proposed majorization minimization ($\mm$ and $\amm$) methods for distributed PGO on the simulated \textsf{\small Cube} datasets and a number of 2D and 3D SLAM benchmark datasets \cite{rosen2016se}. We also make comparisons with the distributed Gauss-Seidel ($\dgs$) method in \cite{choudhary2017distributed}, which is the state-of-the-art method for distributed PGO. We use the certifiably-correct algorithm SE-Sync \cite{rosen2016se} to provide the ground truth and globally optimal objective for all the datasets. All the experiments have been performed on a laptop with an Intel i7-8750H CPU and 32GB of RAM running Ubuntu 18.04 and using g++ 7.8 as C++ compiler. For both $\mm$ and $\amm$ methods, $\xi$ in \cref{eq::Gammaxi} is chosen to be $0.001$, and the $\dgs$ method uses the default settings.

\vspace{-0.35em}
\subsection{\textsf{\small Cube} Datasets}
In this section, we evaluate the convergence of the $\mm$ and $\amm$ methods on $20$ simulated \textsf{\small Cube} datasets with $5$, $10$ and $50$ robots.

In the experiments, a {\sf\small Cube} dataset (\cref{fig::cube}(a)) has $12 \times 12 \times 12$ cube grids with $1$ m side length, and a path of $3600$ poses along the rectilinear edge of the cube grid, and odometric measurements between all the pairs of sequential poses, and loop-closure measurements between nearby but non-sequential poses that are randomly available with a probability of $0.1$. We generate the odometric and loop-closure measurements according to the noise models in \cite{rosen2016se} with an expected translational RMSE  of $\sigma_t=0.02$ m and an expected angular RMSE of $\sigma_R=0.02\pi$ rad.  

The results of a maximum of $1000$ iterations are as shown in \cref{fig::cube}, which has the suboptimality gap $f-f^*$ and the relative suboptimality gap $(f-f^*)/f^*$ in (b) and (c), respectively. In \cref{fig::cube}, $f$ is the objective attained by the  the $\mm$ and $\amm$ methods and $f^*$ is the globally optimal objective attained by SE-Sync \cite{rosen2016se}. It can be seen from \cref{fig::cube} that both $\mm$ and $\amm$ methods have better convergence as the number of robots decreases, which is not surprising since $E(X|X^{(k)})$ in \cref{eq::E,eq::Easum} results in a tighter approximation of distributed PGO in \cref{eq::obj,eq::objM} with fewer robots. In addition, the $\amm$ method always outperforms the $\mm$ method in terms of the convergence, which suggests that Nesterov's method accelerates distributed PGO. In particular, it should be noted that the $\amm$ method with $50$ robots converges almost faster than the $\mm$ method with $5$ robots, which further indicates that the $\amm$ method is well suited for distributed PGO considering the fact that no theoretical guarantees are compromised and only limited extra computation is introduced in acceleration.   

\vspace{-0.15em}
\subsection{SLAM Benchmark Datasets}
In this section, we compare the $\mm$ and $\amm$ methods with the distributed Gauss-Seidel ($\dgs$) method \cite{choudhary2017distributed}, which is the state-of-the-art method for distributed PGO on a number of 2D and 3D SLAM benchmark datasets. It should be noted that originally the $\mm$ and $\amm$ methods and the $\dgs$ method adopt different algorithms to initialize the rotation $R\in SO(d)^n$, and the $\mm$ and $\amm$ methods initialize the translation $t\in\R^{d\times n}$, whereas the $\dgs$ method does not. Therefore, in order to make the comparisons fair, we initialize the $\mm$ and $\amm$ methods and the $\dgs$ method with the centralized chordal initialization for both the rotation $R\in SO(d)^n$ and the translation $t\in\R^{d\times n}$. 

In the experiments, the $\dgs$ method is assigned an ordering according to which the poses of each robot are updated, which improves the convergence performance. Even though such an ordering reduces the number of iterations, parts of the robots have to stay idle until poses of the other robots are updated, which might increase the overall computational time in the end, and thus, is not that desirable in distributed PGO. In contrast, the $\mm$ and $\amm$ methods update the poses of all the robots at the same time and no ordering is needed.

The $\mm$ and $\amm$ methods and the $\dgs$ method are evaluated with $10$ robots. The results are shown as \cref{table::2Dcomparison,table::3Dcomparison}, in which $f^*$ is the objective value of the globally optimal objective attained by SE-Sync \cite{rosen2016se} and $f$ is the objective attained by each method with the given number of iterations, i.e., $100$, $250$ and $1000$. For each dataset and each number of iterations, the best and second results are colored in red and blue, respectively. From \cref{table::2Dcomparison,table::3Dcomparison}, the $\amm$ method outperforms the $\dgs$ method \cite{choudhary2017distributed} on all the datasets except the {\sf intel} dataset, for which the chordal initialization is sufficient for one Gauss-Newton step to attain the global optimum. Even though the $\mm$ method is inferior to the $\amm$ method, it still has a better performance on most of the datasets than the $\dgs$ method. Furthermore, the $\mm$ and $\amm$ are theoretically guaranteed to improve the estimates as the number of iterations increases, whereas the $\dgs$ method, which is equivalent to a one-step Gauss-Newton method, is not --- on the {\sf ais2klinik} dataset, the $\dgs$ method has the objective of $250$ iterations greater than that of $100$ iterations --- and as a matter of fact, as discussed in \cite{rosen2014rise}, the convergence of the Gauss-Newton method without stepsize tuning can not be guaranteed. 

\begin{table*}[h]
	\renewcommand{\arraystretch}{1.3}
	\centering
	\begin{tabular}{|c||c|c|c||c|c|c|c|c|}
		\hline			
		\multirow{2}{*}{Dataset}&\multirow{2}{*}{\# poses}&\multirow{2}{*}{\# edges}&\multirow{2}{*}{$f^*$} &\multirow{2}{*}{\# iterations} &\multicolumn{3}{c|}{$f$}\\
		\cline{6-8}
		&& &&  &{$\mm$ [ours]} &{$\amm$ [ours]} & {$\dgs$\cite{choudhary2017distributed}} \\
		\hline\hline
		\multirow{3}{*}{\sf ais2klinik} &\multirow{3}{*}{$15115$}&\multirow{3}{*}{$16727$}&\multirow{3}{*}{$1.885\times 10^2$}&100&$\color{blue}2.012\times 10^2$ &$\color{red}1.982\times 10^2$ &$8.646\times 10^2$\\
		\cline{5-8}
		& & & & 250 &$\color{blue}1.992\times 10^2$ &$\color{red}1.962\times 10^2$ &$9.315\times 10^2$ \\
		\cline{5-8}
		& & & & 1000&$\color{blue}1.961\times 10^2$ &$\color{red}1.933\times 10^2$ &$3.350\times 10^2$ \\
		\hline
		\hline
		\multirow{3}{*}{\sf city} &\multirow{3}{*}{$10000$}&\multirow{3}{*}{$20687$}&\multirow{3}{*}{$6.386\times10^2$}&100&$\color{blue}6.556\times 10^2$ &$\color{red}6.524\times 10^2$ &$7.989\times 10^2$\\
		\cline{5-8}
		& & & & 250 &$\color{blue}6.529\times 10^2$ &$\color{red}6.484\times 10^2$ &$7.055\times 10^2$ \\
		\cline{5-8}
		& & & & 1000&$\color{blue}6.473\times 10^2$ &$\color{red}6.418\times 10^2$ &$6.562\times 10^2$ \\
		\hline
		\hline
		\multirow{3}{*}{\sf CSAIL} &\multirow{3}{*}{$1045$}&\multirow{3}{*}{$1172$}&\multirow{3}{*}{$3.170\times10^1$}&100&$\color{blue}3.170\times 10^1$ &$\color{red}3.170\times 10^1$ &$3.248\times 10^1$\\
		\cline{5-8}
		& & & & 250 &$\color{blue}3.170\times 10^1$ &$\color{red}3.170\times 10^1$ &$3.179\times 10^1$ \\
		\cline{5-8}
		& & & & 1000&$\color{blue}3.170\times 10^1$ &$\color{red}3.170\times 10^1$ &$3.171\times 10^1$ \\
		\hline
		\hline
		\multirow{3}{*}{\sf M3500} &\multirow{3}{*}{$3500$}&\multirow{3}{*}{$5453$}&\multirow{3}{*}{$1.939\times10^2$}&100&$\color{blue}1.952\times 10^2$ &$\color{red}1.947\times 10^2$ &$1.956\times 10^2$\\
		\cline{5-8}
		& & & & 250 &$1.947\times 10^2$ &$\color{red}1.944\times 10^2$ &$\color{blue}1.946\times 10^2$ \\
		\cline{5-8}
		& & & & 1000&$1.943\times 10^2$ &$\color{red}1.940\times 10^2$ &$\color{blue}1.943\times 10^2$ \\
		\hline
		\hline
		\multirow{3}{*}{\sf intel} &\multirow{3}{*}{$1728$}&\multirow{3}{*}{$2512$}&\multirow{3}{*}{$5.235\times10^1$}&100&$5.257\times 10^1$ &$\color{red}5.252\times 10^1$ &$5.255\times 10^1$\\
		\cline{5-8}
		& & & & 250 &$5.252\times 10^1$ &$\color{blue}5.248\times 10^1$ &$\color{red}5.244\times 10^1$ \\
		\cline{5-8}
		& & & & 1000&$5.243\times 10^1$ &$\color{blue}5.240\times 10^1$ &$\color{red}5.238\times 10^1$ \\
		\hline
		\hline
		\multirow{3}{*}{\sf MITb} &\multirow{3}{*}{$808$}&\multirow{3}{*}{$827$}&\multirow{3}{*}{$6.115\times10^1$}&100&$\color{blue}6.347\times 10^1$ &$\color{red}6.228\times 10^1$ &$9.244\times 10^1$\\
		\cline{5-8}
		& & & & 250 &$\color{blue}6.220\times 10^1$ &$\color{red}6.153\times 10^1$ &$7.453\times 10^1$ \\
		\cline{5-8}
		& & & & 1000&$\color{blue}6.136\times 10^1$ &$\color{red}6.117\times 10^1$ &$6.887\times 10^1$ \\
		\hline
	\end{tabular}

	\caption{Results of the 2D SLAM Benchmark datasets with $10$ robots, in which the best and second results are colored in red and blue, respectively.}	\label{table::2Dcomparison}
	\vspace{1.5em}
	\centering
	\begin{tabular}{|c||c|c|c||c|c|c|c|c|}
		\hline			
		\multirow{2}{*}{Dataset}&\multirow{2}{*}{\# poses}&\multirow{2}{*}{\# edges}&\multirow{2}{*}{$f^*$} &\multirow{2}{*}{\# iterations} &\multicolumn{3}{c|}{$f$}\\
		\cline{6-8}
		&& &&  &{$\mm$ [ours]} &{$\amm$ [ours]} & {$\dgs$\cite{choudhary2017distributed}} \\
		\hline\hline
		\multirow{3}{*}{\sf sphere} &\multirow{3}{*}{$2500$}&\multirow{3}{*}{$4949$}&\multirow{3}{*}{$1.687\times 10^3$}&100&$1.691\times 10^3$ &$\color{red}1.687\times10^3$ &$\color{blue}1.688\times 10^3$\\
		\cline{5-8}
		& & & & 250 &$\color{blue}1.687\times 10^3$ &$\color{red}1.687\times 10^3$ &$1.688\times 10^3$ \\
		\cline{5-8}
		& & & & 1000&$\color{blue}1.687\times 10^3$ &$\color{red}1.687\times 10^3$ &$1.687\times 10^3$ \\
		\hline
		\hline
		\multirow{3}{*}{\sf torus} &\multirow{3}{*}{$5000$}&\multirow{3}{*}{$9048$}&\multirow{3}{*}{$2.423\times10^4$}&100&$\color{blue}2.424\times10^4$ &$\color{red}2.423\times10^4$ &$2.425\times 10^4$\\
		\cline{5-8}
		& & & & 250 &$\color{blue}2.423\times10^4$ &$\color{red}2.423\times10^4$ &$2.425\times 10^4$ \\
		\cline{5-8}
		& & & & 1000&$\color{blue}2.423\times10^4$ &$\color{red}2.423\times10^4$ &$2.424\times 10^4$ \\
		\hline
		\hline
		\multirow{3}{*}{\sf grid} &\multirow{3}{*}{$8000$}&\multirow{3}{*}{$22236$}&\multirow{3}{*}{$8.432\times10^4$}&100&$8.433\times10^4$ &$\color{red}8.432\times10^4$ &$\color{blue}8.433\times 10^4$\\
		\cline{5-8}
		& & & & 250 &$\color{blue}8.432\times10^4$ &$\color{red}8.432\times10^4$ &$8.433\times 10^4$ \\
		\cline{5-8}
		& & & & 1000&$\color{blue}8.432\times10^4$ &$\color{red}8.432\times10^4$ &$8.433\times 10^4$ \\
		\hline
		\hline
		\multirow{3}{*}{\sf garage} &\multirow{3}{*}{$1661$}&\multirow{3}{*}{$6275$}&\multirow{3}{*}{$1.263\times10^0$}&100&$\color{blue}1.279\times10^0$ &$\color{red}1.275\times10^0$ &$1.319\times 10^0$\\
		\cline{5-8}
		& & & & 250 &$\color{blue}1.274\times10^0$ &$\color{red}1.270\times10^0$ &$1.287\times 10^0$ \\
		\cline{5-8}
		& & & & 1000&$\color{blue}1.269\times10^0$ &$\color{red}1.266\times10^0$ &$1.273\times 10^0$ \\
		\hline
		\hline
		\multirow{3}{*}{\sf cubicle} &\multirow{3}{*}{$5750$}&\multirow{3}{*}{$16869$}&\multirow{3}{*}{$7.171\times10^2$}&100&$\color{blue}7.228\times10^2$ &$\color{red}7.204\times10^2$ &$7.317\times 10^2$\\
		\cline{5-8}
		& & & & 250 &$\color{blue}7.206\times10^2$ &$\color{red}7.189\times10^2$ &$7.231\times 10^2$ \\
		\cline{5-8}
		& & & & 1000&$\color{blue}7.185\times10^2$ &$\color{red}7.176\times10^2$ &$7.205\times 10^2$ \\
		\hline
		\hline
		\multirow{3}{*}{\sf rim} &\multirow{3}{*}{$10195$}&\multirow{3}{*}{$29743$}&\multirow{3}{*}{$5.461\times10^3$}&100&$\color{blue}5.779\times10^3$ &$\color{red}5.674\times10^3$ &$6.114\times 10^2$\\
		\cline{5-8}
		& & & & 250 &$\color{blue}5.695\times10^3$ &$\color{red}5.573\times10^3$ &$6.035\times 10^3$ \\
		\cline{5-8}
		& & & & 1000&$\color{blue}5.549\times10^3$ &$\color{red}5.486\times10^3$ &$5.932\times 10^3$ \\
		\hline
	\end{tabular}
	\caption{Results of the 3D SLAM Benchmark datasets with $10$ robots, in which the best and second results are colored in red and blue, respectively.}	\label{table::3Dcomparison}
	\vspace{-1.5em}
\end{table*}

\section{Conclusion}\label{section::conclusion}
In this paper, we have presented majorization minimization methods for distributed PGO that has important applications in multi-robot SLAM. We have proved that our proposed methods are guaranteed to converge to first-order critical points under mild conditions. Furthermore, we have accelerated majorization minimization methods for distributed PGO with Nesterov's method and there is no compromise of theoretical guarantees in acceleration. We have also presented majorization minimization methods for the distributed chordal initialization that converge quadratically. The efficacy of this work has been validated through applications on a number of 2D and 3D SLAM datasets and comparisons with existing state-of-the-art method \cite{choudhary2017distributed}, which indicates that our proposed methods converge faster and result in better solutions to distributed PGO.

\bibliographystyle{IEEEtran}
\bibliography{mybib}

\setcounter{secnumdepth}{2}
\setcounter{section}{0}
\numberwithin{equation}{section}
\renewcommand\thesection{Appendix \Alph{section}}
\section{Proof of \cref{prop::G}}\label{appendix::A}
\renewcommand\thesection{\Alph{section}}
\vspace{0.75em}
For notational simplicity, we introduce
\begin{equation}\label{eq::Ft}
F_{ij}^{t,\alpha\beta}(X)=\frac{1}{2}\|R_i^\alpha \nt_{ij}^{\alpha\beta}+t_i^\alpha -t_j^\beta\|^2,
\end{equation}
\begin{equation}\label{eq::FR}
F_{ij}^{R,\alpha\beta}(X)=\frac{1}{2}\|R_i^\alpha \nR_{ij}^{\alpha\beta} -R_j^\beta\|^2,
\end{equation}
\begin{multline}\label{eq::Et}
E_{ij}^{t,\alpha\beta}(X|X^{(k)})=\|R_i^{\alpha}\nt_{ij}^{\alpha\beta}+t_i^\alpha-\frac{1}{2}R_i^{\alpha(k)}\nt_{ij}^{\alpha\beta}-\frac{1}{2}t_i^{\alpha(k)}-\\
\frac{1}{2}t_j^{\beta(k)}\|^2+\|t_j^\beta-\frac{1}{2}R_i^{\alpha(k)}\nt_{ij}^{\alpha\beta}-\frac{1}{2}t_i^{\alpha(k)}-\frac{1}{2}t_j^{\beta(k)}\|^2,
\end{multline}
\begin{multline}\label{eq::ER}
E_{ij}^{R,\alpha\beta}(X|X^{(k)})=\|R_i^\alpha\nR_{ij}^{\alpha\beta}-\frac{1}{2}R_i^{\alpha(k)}\nR_{ij}^{\alpha\beta}-\frac{1}{2}R_j^{\beta(k)}\|^2+\\
\|R_j^\beta-\frac{1}{2}R_i^{\alpha(k)}\nR_{ij}^{\alpha\beta}-\frac{1}{2}R_j^{\beta(k)}\|^2,
\end{multline}
\begin{multline}\label{eq::Faa}
F^\alpha(X)=\sum_{(i,j)\in \aEE^{\alpha\alpha}}\frac{1}{2}\Big[\kappa_{ij}^{\alpha\alpha}\|R_i^\alpha \nR_{ij}^{\alpha\alpha} -R_j^\alpha\|^2 +\\ \tau_{ij}^{\alpha\alpha}\|R_i^\alpha \nt_{ij}^{\alpha\alpha}+t_i^\alpha - t_j^\alpha\|^2\Big],
\end{multline}
which will be used to prove \cref{prop::G}.

From \cref{eq::Ft,eq::FR}, it can be shown that
\begin{subequations}\label{eq::nabla_FR}
\begin{equation}
\nabla_{t_i^\alpha} F_{ij}^{t,\alpha\beta}(X) = R_i^\alpha\nt_{ij}^{\alpha\beta}+t_i^\alpha- t_j^\beta,
\end{equation}
\begin{equation}
\nabla_{R_i^\alpha} F_{ij}^{t,\alpha\beta}(X) =  \left(R_i^\alpha\nt_{ij}^{\alpha\beta}+t_i^\alpha- t_j^\beta\right)\nt{\vphantom{t}_{ij}^{\alpha\beta}}^\transpose,
\end{equation}
\begin{equation}
\nabla_{t_j^\beta} F_{ij}^{t,\alpha\beta}(X) = t_j^\beta-R_i^\alpha\nt_{ij}^{\alpha\beta}-t_i^\alpha.
\end{equation}
\end{subequations}
\begin{subequations}\label{eq::nabla_Ft}
	\begin{equation}
	\nabla_{R_i^\alpha} F_{ij}^{R,\alpha\beta}(X) = R_i^\alpha- R_j^\beta\nR{\vphantom{R}_{ij}^{\alpha\beta}}^\transpose,
	\end{equation}
	\begin{equation}
	\nabla_{R_j^\beta} F_{ij}^{R,\alpha\beta}(X) =  R_j^\beta-R_i^\alpha\nR_{ij}^{\alpha\beta},
	\end{equation}
\end{subequations}
Furthermore, since $F_{ij}^{R,\alpha\beta}(X)$ and $F_{ij}^{t,\alpha\beta}(X)$ are only related with $(t_i^\alpha,\,R_i^\alpha)$ and $(t_j^\beta,\,R_j^\beta)$, then $\nabla F_{ij}^{R,\alpha\beta}(X)$ and $\nabla F_{ij}^{t,\alpha\beta}(X)$ are well defined by \cref{eq::nabla_FR,eq::nabla_Ft}, respectively. 

From \cref{eq::ER,eq::Et,eq::nabla_FR,eq::nabla_Ft}, a tedious but straightforward mathematical manipulation indicates that 
\begin{equation}
\nonumber
\begin{aligned}
&E_{ij}^{t,\alpha\beta}(X|X^{(k)})\\
=&\|(R_i^\alpha -R_i^{\alpha(k)})\nt_{ij}^{\alpha\beta}+t_i^\alpha-t_i^{\alpha(k)} \|^2+\\
&\|t_j^\beta-t_j^{\beta(k)}\|^2+\\
&\innprod{\nabla_{t_i^\alpha} F_{ij}^{R,\alpha\beta}(X^{(k)})}{t_i^\alpha-t_j^{\alpha(k)}}+\\
&\innprod{\nabla_{R_i^\alpha} F_{ij}^{t,\alpha\beta}(X^{(k)})}{R_i^\alpha-R_i^{\alpha(k)}}+\\
&\innprod{\nabla_{t_j^\beta} F_{ij}^{t,\alpha\beta}(X^{(k)})}{t_j^\beta-t_j^{\beta(k)}}+\\
&F_{ij}^{t,\alpha\beta}(X^{(k)})\\
=&\|(R_i^\alpha -R_i^{\alpha(k)})\nt_{ij}^{\alpha\beta}+t_i^\alpha-t_i^{\alpha(k)} \|^2+\\
&\|t_j^\beta-t_j^{\beta(k)}\|^2+\\
&\innprod{\nabla F_{ij}^{t,\alpha\beta}(X^{(k)})}{X-X^{(k)}}+\\
&F_{ij}^{t,\alpha\beta}(X^{(k)})
\end{aligned}
\end{equation}
and
\begin{equation}
\nonumber
\begin{aligned}
&E_{ij}^{R,\alpha\beta}(X|X^{(k)})\\
=&\|R_i^\alpha -R_i^{\alpha(k)} \|^2+\|R_j^\beta-R_j^{\beta(k)}\|^2+\\
&\innprod{\nabla_{R_i^\alpha} F_{ij}^{R,\alpha\beta}(X^{(k)})}{R_i^\alpha-R_i^{\alpha(k)}}+\\
&\innprod{\nabla_{R_j^\beta} F_{ij}^{R,\alpha\beta}(X^{(k)})}{R_j^\beta-R_j^{\beta(k)}}+\\
&F_{ij}^{R,\alpha\beta}(X^{(k)})\\
=&\|R_i^\alpha -R_i^{\alpha(k)} \|^2+\|R_j^\beta-R_j^{\beta(k)}\|^2+\\
&\innprod{\nabla F_{ij}^{R,\alpha\beta}(X^{(k)})}{X-X^{(k)}}+\\
&F_{ij}^{R,\alpha\beta}(X^{(k)}),
\end{aligned}
\end{equation}
from which we obtain
\begin{subequations}\label{eq::Eobj}
\begin{equation}
E_{ij}^{t,\alpha\beta}(X^{(k)}|X^{(k)})=F_{ij}^{t,\alpha\beta}(X^{(k)}),
\end{equation}
\begin{equation}
E_{ij}^{R,\alpha\beta}(X^{(k)}|X^{(k)})=F_{ij}^{R,\alpha\beta}(X^{(k)}),
\end{equation}
\end{subequations}
\begin{subequations}\label{eq::Egrad}
	\begin{equation}
	\nabla E_{ij}^{t,\alpha\beta}(X^{(k)}|X^{(k)})=\nabla F_{ij}^{t,\alpha\beta}(X^{(k)}),
	\end{equation}
	\begin{equation}
	\nabla E_{ij}^{R,\alpha\beta}(X^{(k)}|X^{(k)})=\nabla F_{ij}^{R,\alpha\beta}(X^{(k)}),
	\end{equation}
\end{subequations}
As a result of \cref{eq::E,eq::obj} and \cref{eq::FR,eq::Ft,eq::Faa,eq::ER,eq::Et},  it can be concluded that
\begin{multline}\label{eq::Fsum}
F(X)=\sum_{\alpha\in \AA} F^\alpha(X) +\\
\sum_{\substack{\alpha,\beta\in \AA,\\\alpha\neq \beta}}\sum_{(i,j)\in \aEE^{\alpha\beta}}\Big[\kappa_{ij}^{\alpha\beta}\cdot F_{ij}^{R,\alpha\beta}(X)+
\tau_{ij}^{\alpha\beta}\cdot F_{ij}^{t,\alpha\beta}(X)\Big]
\end{multline}
and
\begin{multline}\label{eq::Esum}
E(X|X^{(k)})=\sum_{\alpha\in \AA} F^\alpha(X) +\\
\sum_{\substack{\alpha,\beta\in \AA,\\\alpha\neq \beta}}\sum_{(i,j)\in \aEE^{\alpha\beta}}\Big[\kappa_{ij}^{\alpha\beta}\cdot E_{ij}^{R,\alpha\beta}(X|X^{(k)})+\\
\tau_{ij}^{\alpha\beta}\cdot E_{ij}^{t,\alpha\beta}(X|X^{(k)})\Big].
\end{multline}
Applying \cref{eq::Egrad,eq::Eobj} to \cref{eq::Fsum,eq::Esum}, we obtain
\begin{equation}\label{eq::EF}
E(X^{(k)}|X^{(k)})=F(X^{(k)})
\end{equation}
and
\begin{equation}\label{eq::gEF}
\nabla E(X^{(k)}|X^{(k)})=\nabla F(X^{(k)}).
\end{equation}
It can be seen from \cref{eq::E} that $E(X|X^{(k)})$ is a sum of squares, which suggests that it is equivalent to
\begin{multline}\label{eq::Eprox}
E(X|X^{(k)})=
\frac{1}{2}\innprod{\nH(X-X^{(k)})}{X-X^{(k)}}+\\
\innprod{\nabla E(X^{(k)}|X^{(k)})}{X-X^{(k)}}+
E(X^{(k)}|X^{(k)}),
\end{multline}
in which $\nH=\nabla^2 E(X^{(k)}|X^{(k)})\in \R^{(d+1)n\times (d+1)n}$ is the Euclidean Hessian of $E(X|X^{(k)})$ at $X^{(k)}$.
Substituting \cref{eq::EF,eq::gEF} into \cref{eq::Eprox}, we obtain \cref{eq::EM}
\begin{multline}
\nonumber
E(X|X^{(k)})=
\frac{1}{2}\innprod{\nH(X-X^{(k)})}{X-X^{(k)}}+\\
\innprod{\nabla F(X^{(k)})}{X-X^{(k)}}+
F(X^{(k)}).
\end{multline}
From \cref{eq::Ea,eq::Easum}, it can be shown that $\nH=\nabla^2 E(X^{(k)}|X^{(k)})$ is a block diagonal matrix
\begin{equation}
\nonumber
\nH\triangleq\mathrm{diag}\big\{\nH^1,\,\cdots,\,\nH^A\big\}\in \R^{(d+1)n\times (d+1)n}
\end{equation}
in which $\nH^\alpha=\nabla^2 E^\alpha(X^{\alpha(k)}|X^{(k)})\in \R^{(d+1)n_\alpha\times (d+1)n_\alpha}$ is Euclidean Hessian of $E^\alpha(X^{\alpha}|X^{(k)})$ at $X^{\alpha (k)}$.  More explicitly, if we let
$$\aEE_{i-}^{\alpha\beta}\triangleq \{(i,\,j)|(i,\,j)\in \aEE^{\alpha\beta}\},$$
$$\aEE_{i+}^{\alpha\beta}\triangleq \{(i,\,j)|(j,\,i)\in \aEE^{\beta\alpha}\},$$
$$\NN_{i-}^{\alpha}\triangleq \{\beta\in \AA|\exists(i,\,j)\in \aEE^{\alpha\beta}\text{ and }\beta\neq\alpha\},$$
$$\NN_{i+}^{\alpha}\triangleq \{\beta\in \AA|\exists(j,\,i)\in \aEE^{\beta\alpha}\text{ and }\beta\neq\alpha\},$$ $$\EE_i^{\alpha\beta}\triangleq\aEE_{i-}^{\alpha\beta}\cup\aEE_{i+}^{\alpha\beta},$$
$$\NN_i^\alpha\triangleq\NN_{i-}^\alpha\cup\NN_{i+}^\alpha, $$  
and $\kappa_{ji}^{\beta\alpha}=\kappa_{ij}^{\alpha\beta}$ and $\tau_{ji}^{\beta\alpha}=\tau_{ij}^{\alpha\beta}$, it is straightforward to show that $\nH^\alpha\in \R^{(d+1)n_\alpha\times (d+1)n_\alpha}$ is a sparse matrix 
\begin{equation}
\nH^\alpha=\begin{bmatrix}
\nH^{\tau,\alpha}\hphantom{{}^\transpose} & \nH^{\nu,\alpha}\\
\nH{\vphantom{H}^{\nu,\alpha}}^\transpose & \nH^{\kappa,\alpha}
\end{bmatrix}
\end{equation}
and $\nH^{\tau,\alpha}\in \R^{n_\alpha\times n_\alpha}$, $\nH^{\nu,\alpha}\in\R^{n_\alpha\times dn_\alpha}$ and $\nH^{\kappa,\alpha}\in\R^{dn_\alpha\times dn_\alpha}$ are defined as
\allowdisplaybreaks
\begin{align}
&\big[\nH^{\tau,\alpha}\big]_{ij}=\begin{cases}
\sum\limits_{e\in \EE_i^{\alpha\alpha}}\tau_{e}^{\alpha\alpha}+\\
\quad\quad\;\;\sum\limits_{\substack{\beta\in \NN_i^\alpha}}\sum\limits_{e\in \EE_i^{\alpha\beta}}2\tau_e^{\alpha\beta}, & i=j\\
-\tau_{ij}^{\alpha\alpha}, & (i,j)\in \aEE^{\alpha\alpha},\\
-\tau_{ji}^{\alpha\alpha}, & (j,i)\in \aEE^{\alpha\alpha},\\
0,& \text{otherwise,}
\end{cases}\\[0.8em]
&\big[\nH^{\nu,\alpha}\big]_{ij}=\begin{cases}
\sum\limits_{e\in \EE_{i-}^{\alpha\alpha}}\tau_{e}^{\alpha\alpha}\nt{_{e}^{\alpha\alpha}}^\transpose+\\
\;\sum\limits_{\substack{\beta\in \NN_{i-}^\alpha}}\sum\limits_{e\in \EE_{i-}^{\alpha\beta}}2\tau_e^{\alpha\beta}\nt{\vphantom{t}_{e}^{\alpha\beta}}^\transpose, & i=j\\
-\tau_{ji}^{\alpha\alpha}\nt{_{ji}^{\alpha\alpha}}^\transpose, & (j,i)\in \aEE^{\alpha\alpha},\\
0,& \text{otherwise,}
\end{cases}\\[0.8em]
&\big[\nH^{\kappa,\alpha}\big]_{ij}=\begin{cases}
\sum\limits_{e\in \EE_i^{\alpha\alpha}}\kappa_{e}^{\alpha\alpha}\cdot\I+\\
\quad\sum\limits_{e\in \EE_{i-}^{\alpha\alpha}}\tau_{e}^{\alpha\alpha}\cdot\nt_{e}^{\alpha\alpha}\nt{\vphantom{t}_{e}^{\alpha\alpha}}^\transpose+\\
\;\sum\limits_{\substack{\beta\in \NN_i^\alpha}}2\big(\sum\limits_{e\in \EE_i^{\alpha\beta}}\kappa_e^{\alpha\beta}\cdot\I+\\
\quad\quad\sum\limits_{e\in \EE_{i-}^{\alpha\beta}}\;\tau_{e}^{\alpha\beta}\cdot\nt_{e}^{\alpha\beta}\nt{\vphantom{t}_{e}^{\alpha\beta}}^\transpose\big), & i=j\\
-\kappa_{ij}^{\alpha\alpha}\cdot\nR_{ij}^{\alpha\alpha}, & (i,j)\in \aEE^{\alpha\alpha},\\
-\kappa_{ji}^{\alpha\alpha}\cdot\nR{\vphantom{R}_{ij}^{\alpha\alpha}}^\transpose, & (j,i)\in \aEE^{\alpha\alpha},\\
0,& \text{otherwise,}
\end{cases}
\end{align}
in which $\big[\nH^{\tau,\alpha}\big]_{ij}\in \R$, $\big[\nH^{\nu,\alpha}\big]_{ij}\in \R^{1\times d}$ and $\big[\nH^{\kappa,\alpha}\big]_{ij}\in \R^{d\times d}$.

From \cref{eq::obj,eq::inequality,eq::E}, it can be concluded that $E(X|X^{(k)})\geq F(X)$ for any $X\in \R^{d\times (d+1)n}$ in which ``$=$'' holds if $X=X^{(k)}$. Furthermore, as a result of \cref{eq::objM}, it is possible to rewrite $F(X)$ as 
\begin{multline}\label{eq::F}
F(X)=
\frac{1}{2}\innprod{\nM(X-X^{(k)})}{X-X^{(k)}}+\\
\innprod{\nabla F(X^{(k)})}{X-X^{(k)}}+
F(X^{(k)}),
\end{multline}
and Eqs. \eqref{eq::EM} and \eqref{eq::F} and $E(X|X^{(k)})\geq F(X)$ suggest $\nH\succeq \nM$, which completes the proof.

\renewcommand\thesection{Appendix \Alph{section}}
\section{Proof of \cref{prop::mm}}\label{appendix::B}
\renewcommand\thesection{\Alph{section}}

\vspace{0.75em}

\noindent\textbf{Proof of \ref{prop::mm1}.\;} From \cref{prop::G} and \cref{eq::G,eq::Gammaxi,eq::EM}, we obtain 
\begin{equation}\label{eq::inG}
G(X|X^{(k)})\geq E(X|X^{(k)}) \geq F(X),
\end{equation}
in which ``$=$'' holds if $X=X^{(k)}$. Furthermore, as a result of \cref{eq::inG} and \cref{assumption::local}, it can be concluded that
\begin{multline}\label{eq::order}
F(X^{(k+1)}) \leq G(X^{(k+1)}|X^{(k)}) \leq\\
G(X^{(k)}|X^{(k)}) =  F(X^{(k)}),
\end{multline}
which suggests that $F(X^{(k)})$ is non-increasing. The proof is completed.

\vspace{1.2em}

\noindent\textbf{Proof of \ref{prop::mm2}.\;}  From \ref{prop::mm1} of \cref{prop::mm}, it has been proved that $F(X^{(k)})$ is non-increasing. Moreover, from \cref{eq::obj}, $F(X^{(k)})\geq 0$, i.e., $F(X^{(k)})$ is bounded below. As a result, there exists $F^{\infty} \in \R $ such that $F(X^{(k)})\rightarrow F^{\infty}$. The proof is completed.

\vspace{1.2em}

\noindent\textbf{Proof of \ref{prop::mm3}.\;} From \cref{eq::order}, we obtain
\begin{multline}\label{eq::bound}
F(X^{(k)})- F(X^{(k+1)}) \geq \\
G(X^{(k+1)}|X^{(k)}) - F(X^{(k+1)})
\end{multline}
From Eqs. \eqref{eq::G} and \eqref{eq::F}, we obtain
\begin{multline}\label{eq::GF}
G(X^{(k+1)}|X^{(k)}) - F(X^{(k+1)}) =\\
 \frac{1}{2}\innprod{ (\nGamma-\nM)(X^{(k+1)}-X^{(k)})}{X^{(k+1)}-X^{(k)}}.
\end{multline}
If $\nGamma \succ \nM$, there exists $\delta > 0$ such that $\nGamma \succeq \nM + \delta\cdot \I$, and as a result of \cref{eq::GF}, we obtain
\begin{multline}\label{eq::GFg}
G(X^{(k+1)}|X^{(k)}) - F(X^{(k+1)}) \geq\\
\frac{\delta}{2}\|X^{(k+1)}-X^{(k)}\|^2.
\end{multline}
From \cref{eq::bound,eq::GFg}, it can be concluded that
\begin{equation}
\nonumber
F(X^{(k)}) - F(X^{(k+1)}) \geq
\frac{\delta}{2}\|X^{(k+1)}-X^{(k)}\|^2,
\end{equation}
and since $F(X^{(k+1)})\leq F(X^{(k)})$, we further obtain
\begin{equation}\label{eq::dF}
\left|F(X^{(k)}) - F(X^{(k+1)})\right| \geq
\frac{\delta}{2}\|X^{(k+1)}-X^{(k)}\|^2.
\end{equation}
From \ref{prop::mm2} of \cref{prop::mm}, we obtain
\begin{equation}\label{eq::dF0}
\left|F(X^{(k)}) - F(X^{(k+1)})\right|\rightarrow 0,
\end{equation}
and \cref{eq::dF,eq::dF0} suggest
\begin{equation}
\|X^{(k+1)}-X^{(k)}\| \rightarrow 0.
\end{equation}

\vspace{1.2em}

\noindent\textbf{Proof of \ref{prop::mm4}.\;} It is straightforward to show that the Riemannian gradient $\grad\, F(X)$ takes the form as
\begin{equation}
\nonumber
\grad\, F(X) =\begin{bmatrix}
\grad_1 F(X) & \cdots  & \grad_A F(X)
\end{bmatrix}\in T_X \XX,
\end{equation}
in which
\begin{equation}
\nonumber
\grad_\alpha F(X)=
\begin{bmatrix}
\grad_{t^\alpha} F(X) & \grad_{R^\alpha} F(X)
\end{bmatrix}\in
T_{X^\alpha}\XX^\alpha.
\end{equation}
and $$T_{X^\alpha}\XX^\alpha\triangleq\R^{d\times n_\alpha}\times T_{R^\alpha} SO(d)^{n_\alpha}.$$
From \cite{absil2009optimization,rosen2016se}, it can be further shown that $\grad_{t^\alpha} F(X)$ and $\grad_{R^\alpha} F(X)$ are
\begin{equation}\label{eq::grad_t}
\grad_{t^\alpha} F(X) = \nabla_{t^\alpha} F(X)
\end{equation}
and
\begin{multline}\label{eq::grad_R}
\grad_{R^\alpha} F(X) = \nabla_{R^\alpha} F(X)-\\
R^\alpha\, \mathrm{SymBlockDiag}_d^\alpha(\nabla_{R^\alpha} F(X)).
\end{multline}
In \cref{eq::grad_R}, $\mathrm{SymBlockDiag}_d^\alpha: \R^{dn_\alpha\times dn_\alpha}\rightarrow \R^{dn_\alpha\times dn_\alpha}$ is a linear operator
\begin{equation}\label{eq::sym}
\mathrm{SymBlockDiag}_d^\alpha(Z)\triangleq\frac{1}{2}\mathrm{BlockDiag}_d^\alpha(Z+Z^\transpose),
\end{equation} 
in which $\mathrm{BlockDiag}_d^\alpha: \R^{dn_\alpha \times dn_\alpha}\rightarrow \R^{dn_\alpha \times dn_\alpha}$ extracts the $d\times d$-block diagonals of a matrix, i.e.,
\begin{equation}
\nonumber
\mathrm{BlockDiag}_d^\alpha(Z)\triangleq\begin{bmatrix}
Z_{11} & & \\
&\ddots &\\
& & Z_{n_\alpha n_\alpha}
\end{bmatrix}\in \R^{dn_\alpha \times dn_\alpha}.
\end{equation}
As a result of \cref{eq::grad_t,eq::grad_R,eq::sym}, there exists a linear operator $$\QQ_X:\R^{d(n+1)\times d(n+1)}\rightarrow T_X\XX$$ that depends on $X\in \XX$ such that
\begin{equation}\label{eq::gradF}
\grad\, F(X)=\QQ_X(\nabla F(X)).
\end{equation}
From \cref{eq::G} and $\nabla F(X)=X\nM$, it is straightforward to show that
\begin{equation}\label{eq::gradG}
\begin{aligned}
&\nabla G(X^{(k+1)}|X^{(k)})\\
=&\nabla F(X^{(k)})+(X^{(k+1)}-X^{(k)})\nGamma\\
=&\nabla F(X^{(k+1)})+(X^{(k+1)}-X^{(k)})(\nGamma-\nM).
\end{aligned}
\end{equation}
It should be noted that \cref{eq::gradF} applies to any functions on $\XX$, and as a result of \cref{eq::gradF,eq::gradG}, we obtain 
\begin{multline}\label{eq::gradGX}
\grad\,G(X^{(k+1)}|X^{(k)})=\grad\, F(X^{(k+1)}) + \\
\QQ_{X^{(k+1)}}\left((X^{(k+1)}-X^{(k)})(\nGamma-\nM)\right).
\end{multline}
From \cref{assumption::local}, it is assumed that 
$$\grad\, G^\alpha(X^{\alpha(k+1)}|X^{(k)})=\0,$$
and as a result of \cref{eq::Gsub,eq::Ga}, it is by definition that
\begin{multline}
\nonumber
\grad\, G(X|X^{(k)})=\\
\begin{bmatrix}
\grad\, G^1(X^1|X^{1(k)}) & \cdots & \grad\, G^A(X^A|X^{A(k)})
\end{bmatrix}
\end{multline}
which suggests
\begin{equation}\label{eq::gradG0}
\grad\,G(X^{(k+1)}|X^{(k)})=\0.
\end{equation}
From \cref{eq::gradG0,eq::gradGX}, we obtain
\begin{equation}
\nonumber
\grad\, F(X^{(k+1)})=\QQ_{X^{(k+1)}}\left((X^{(k)}-X^{(k+1)})(\nGamma-\nM)\right),
\end{equation}
and furthermore,
\begin{equation}\label{eq::gradGF}
\begin{aligned}
&\|\grad\, F(X^{(k+1)})\|\\ 
=&\|\QQ_{X^{(k+1)}}\left((X^{(k+1)}-X^{(k)})(\nGamma-\nM)\right)\|\\
\leq& \|\QQ_{X^{(k+1)}}\|_2\cdot \|\nGamma-\nM\|_2\cdot\|X^{(k)}-X^{(k+1)}\|,
\end{aligned}
\end{equation}
in which $\|\QQ_{X^{(k+1)}}\|_2$ and $\|\nGamma-\nM\|_2$ are the induced $2$-norms of linear operators $\QQ_{X^{(k+1)}}(\cdot)$ and $\nGamma-\nM$, respectively. From \cref{eq::grad_R,eq::grad_t}, it can be concluded that $\QQ_X(\cdot)$ is only related with the rotation $R^1,\,\cdots,\,R^A$ in which $R^\alpha\in SO(d)^{n_\alpha}$, and since $\QQ_X(\cdot)$ depends on $R^1,\,\cdots,\,R^A$ continuously and $SO(d)^{n_\alpha}$ is a compact manifold, $\|\QQ_X\|_2$ is bounded. Furthermore, $\|\nGamma-\nM\|_2$ is also bounded and we have proved that $\|X^{(k+1)}-X^{(k)}\|\rightarrow 0$ if $\nGamma\succ\nM$ in \ref{prop::mm3} of \cref{prop::mm}, and as a result of \cref{eq::gradGF}, we obtain
\begin{equation}\label{eq::G_F}
\|\grad\, F(X^{(k+1)})\|\rightarrow 0.
\end{equation}
if $\nGamma\succ\nM$. From \cref{eq::G_F}, it can be concluded that
\begin{equation}
\nonumber
\grad\, F(X^{(k+1)})\rightarrow \0
\end{equation}
if $\nGamma\succ\nM$. The proof is completed.

\vspace{1.2em}
\noindent\textbf{Proof of \ref{prop::mm5} and \ref{prop::mm6}.\;} From \cref{eq::Gamma,eq::Gammaxi}, we obtain
\begin{equation}
\nonumber
\nGamma=\nH + \xi\cdot\I.
\end{equation}
From \cref{prop::G}, it is known that $\nH\succeq \nM$, and as a result, if $\xi>0$, we obtain
\begin{equation}
\nonumber
\nGamma\succ \nM.
\end{equation}
From \ref{prop::mm3} and \ref{prop::mm4} of \cref{prop::mm}, we obtain $$\|X^{(k+1)}-X^{(k)}\|\rightarrow 0$$
and 
$$\grad\,F(X^{(k)})\rightarrow \0, $$
respectively. The proof is completed.

\renewcommand\thesection{Appendix \Alph{section}}
\section{Proof of \cref{prop::amm}}\label{appendix::C}
\renewcommand\thesection{\Alph{section}}

\vspace{0.75em}
\noindent\textbf{Proof of \ref{prop::amm0}.\;} 
From lines~\ref{line::restart1} to \ref{line::restart2} of \cref{algorithm::amm} and \cref{assumption::local}, it can be concluded that the $\amm$ method has either
\begin{equation}
\nonumber
G^\alpha(X^{\alpha(k+1)}|X^{(k)})\leq \overline{G}^{\alpha (k)}
\end{equation}
or
\begin{equation}
\nonumber
G^\alpha(X^{\alpha(k+1)}|X^{(k)})\leq 
G^\alpha(X^{\alpha(k)}|X^{(k)})=\overline{G}^{\alpha(k)},
\end{equation}
which suggests that
\begin{equation}\label{eq::Gbk}
G^\alpha(X^{\alpha(k+1)}|X^{(k)})\leq \overline{G}^{\alpha (k)}.
\end{equation}
From \cref{eq::Gsub,eq::Ga,eq::Fa}, we obtain
\begin{equation}\label{eq::Gxk}
G(X^{(k+1)}|X^{(k)})=\sum\limits_{\alpha\in \AA} G^{\alpha}(X^{\alpha(k+1)}|X^{(k)})
\end{equation} 
and 
\begin{equation}\label{eq::Fk}
F(X^{(k)})=\sum\limits_{\alpha\in \AA} \overline{G}^{\alpha(k)}.
\end{equation}
From \cref{eq::Gxk,eq::Fk,eq::Gbk}, we obtain
\begin{equation}\label{eq::GFxk}
G(X^{(k+1)}|X^{(k)})\leq F(X^{(k)}).
\end{equation}
From \cref{eq::inG,eq::GFxk}, we obtain
\begin{equation}\label{eq::fcf}
F(X^{(k+1)})\leq G(X^{(k+1)}|X^{(k)})\leq F(X^{(k)}).
\end{equation}
The proof is completed.

\vspace{1.2em}
\noindent\textbf{Proof of \ref{prop::amm1}.\;} From \ref{prop::amm0} of \cref{prop::amm}, it has been proved that $F(X^{(k)})$ is non-increasing. Moreover, from \cref{eq::obj}, $F(X^{(k)})\geq 0$, i.e., $F(X^{(k)})$ is bounded below. As a result, there exists $F^{\infty} \in \R $ such that $F(X^{(k)})\rightarrow F^{\infty}$. The proof is completed.

\vspace{1.2em}

\noindent\textbf{Proof of \ref{prop::amm2}.\;} Similar to the proof of \ref{prop::mm3} of \cref{prop::mm}, from \cref{eq::fcf}, we obtain
\begin{multline}\label{eq::dGk}
F(X^{(k)})- F(X^{(k+1)}) \geq \\
G(X^{(k+1)}|X^{(k)}) - F(X^{(k+1)})
\end{multline}
From Eqs. \eqref{eq::G} and \eqref{eq::F}, we obtain
\begin{multline}
\nonumber
G(X^{(k+1)}|X^{(k)}) - F(X^{(k+1)}) =\\
\frac{1}{2}\innprod{ (\nGamma-\nM)(X^{(k+1)}-X^{(k)})}{X^{(k+1)}-X^{(k)}}.
\end{multline}
If $\nGamma\succ \nM$, there exists $\delta > 0$ such that \cref{eq::GFg} holds, substituting which into the equation above, we obtain
\begin{equation}
\nonumber
G(X^{(k+1)}|X^{(k)}) - F(X^{(k+1)}) \geq
\frac{\delta}{2}\|X^{(k+1)}-X^{(k)}\|^2,
\end{equation}
which suggests
\begin{equation}
\nonumber
F(X^{(k)}) - F(X^{(k+1)}) \geq
\frac{\delta}{2}\|X^{(k+1)}-X^{(k)}\|^2,
\end{equation}
From \ref{prop::amm0} and \ref{prop::amm1} of \cref{prop::amm}, it can be concluded that $F(X^{(k)})-F(X^{(k+1)}) \geq 0$ and $F(X^{(k)})-F(X^{(k+1)}) \rightarrow 0$, and from the equation above, we obtain
\begin{equation}
\nonumber
\|X^{(k+1)}-X^{(k)}\|\rightarrow 0
\end{equation}
if $\nGamma\succ\nM$. The proof is completed.

\vspace{1.2em}

\noindent\textbf{Proof of \ref{prop::amm3}.\;} 
If $G^\alpha(Z^{\alpha(k+1)}|Y^{\alpha(k)})\leq\overline{G}^{\alpha(k)}$, then $X^{\alpha(k+1)}=Z^{\alpha(k+1)}$,
and following a similar procedure of \cref{eq::gradGX}, we obtain
\begin{multline}
\nonumber
\grad\, G^\alpha(X^{\alpha(k+1)}|Y^{\alpha(k)})=\grad_\alpha F(X^{(k+1)})+\\
\QQ_{X^{(k+1)}}^\alpha\left((X^{(k+1)}-Y^{(k)})(\nGamma-\nM)\right),
\end{multline}
in which $\QQ_{X}^\alpha: \R^{d\times dn}\rightarrow \R^{d\times dn_\alpha}$ is a linear operator such that $\QQ_{X}^\alpha(\cdot)$ extracts the $\alpha$-th block of $\QQ_{X}(\cdot)$. From \cref{assump::localZ}, the equation above is simplified to
\begin{multline}\label{eq::gradFz}
\grad_\alpha F(X^{(k+1)})=\\
\QQ_{X^{(k+1)}}^\alpha\left((Y^{(k)}-X^{(k+1)})(\nGamma-\nM)\right).
\end{multline}
From line~\ref{line::nesterov1} of \cref{algorithm::amm}, we obtain 
\begin{equation}\label{eq::Y}
Y^{(k)} = X^{(k)}+\left(X^{(k)}-X^{(k-1)}\right)\cdot \gamma^{(k)},
\end{equation}
in which
\begin{equation}
\nonumber
\gamma^{(k)}=\diag\{\gamma^{1(k)}\cdot\I^1,\,\cdots,\,\gamma^{A(k)}\cdot\I^A\}\in \R^{(d+1)n\times(d+1)n}
\end{equation}
and $\I^\alpha\in \R^{(d+1)n_\alpha\times (d+1)n_\alpha}$ is the identity matrix.
Substituting \cref{eq::Y} into \cref{eq::gradFz}, we obtain
\begin{multline}\label{eq::gradZY}
\!\!\!\!\!\!\grad_\alpha F(X^{(k+1)})=
\QQ_{X^{(k+1)}}^\alpha\!\left((X^{(k)}\!-\!X^{(k+1)})(\nGamma-\nM)\right)+\\
\QQ_{X^{(k+1)}}^\alpha\left((X^{(k)}-X^{(k-1)})\cdot\gamma^{(k)}\cdot(\nGamma-\nM)\right).
\end{multline}
If $G^\alpha(Z^{\alpha(k+1)}|Y^{\alpha(k)})>\overline{G}^{\alpha(k)}$, then
$$X^{\alpha(k+1)}\leftarrow\arg\min\limits_{X^\alpha\in\XX^\alpha }G^\alpha(X^\alpha|X^{(k)}),$$
and we obtain
\begin{multline}
\nonumber
\grad\, G^\alpha(X^{\alpha(k+1)}|X^{(k)})=\grad_\alpha F(X^{(k+1)})+\\
\QQ_{X^{(k+1)}}^\alpha\left((X^{(k+1)}-X^{(k)})(\nGamma-\nM)\right).
\end{multline}
From \cref{assumption::local}, the equation above is simplified to
\begin{multline}\label{eq::gradXX}
\grad_\alpha F(X^{(k+1)})=\\
\QQ_{X^{(k+1)}}^\alpha\left((X^{(k)}-X^{(k+1)})(\nGamma-\nM)\right).
\end{multline}
From \cref{eq::gradXX,eq::gradZY}, it can be concluded that
\begin{multline}
\nonumber
\|\grad_\alpha F(X^{(k+1)})\|\leq\\
\|\QQ_{X^{(k+1)}}^\alpha\left((X^{(k)}-X^{(k+1)})(\nGamma-\nM)\right)\|+\\
\|\QQ_{X^{(k+1)}}^\alpha\left((X^{(k)}-X^{(k-1)})\cdot\gamma^{(k)}\cdot(\nGamma-\nM)\right)\|,
\end{multline}
no matter whether $G^\alpha(Z^{\alpha(k+1)}|Y^{\alpha(k)})\leq\overline{G}^{\alpha(k)}$ or not, which suggests
\begin{multline}\label{eq::grada}
\|\grad_\alpha F(X^{(k+1)})\|\leq\\
\|\QQ_{X^{(k+1)}}^\alpha\|_2\cdot\|\nGamma-\nM\|_2\cdot
\Big(\|X^{(k+1)}-X^{(k)}\|+\\
\|\gamma^{(k)}\|_2\cdot\|X^{(k)}-X^{(k-1)}\|\Big),
\end{multline}
in which $\|\QQ_{X^{(k+1)}}^\alpha\|_2$, $\|\nGamma-\nM\|_2$ and $\|\gamma^{(k)}\|$ are induced $2$-norms. From line~\ref{line::s} of \cref{algorithm::amm}, we obtain $s^{\alpha(k)}\geq1$ and
\begin{equation}\label{eq::gammaa}
\gamma^{\alpha(k)}=\frac{\sqrt{4{s^{\alpha(k)}}^2+1}-1}{2s^{\alpha(k)}}=\frac{2s^{\alpha(k)}}{\sqrt{4{s^{\alpha(k)}}^2+1}+1}\in(0,\,1),
\end{equation}
which suggests $\|\gamma^{(k)}\|_2\in(0,\,1)$. From \cref{eq::grada,eq::gammaa}, it can be concluded that
\begin{multline}\label{eq::gradb}
\|\grad_\alpha F(X^{(k+1)})\|\leq\\
\|\QQ_{X^{(k+1)}}^\alpha\|_2\cdot\|\nGamma-\nM\|_2\cdot
\Big(\|X^{(k+1)}-X^{(k)}\|+\\
\|X^{(k)}-X^{(k-1)}\|\Big),
\end{multline}
Following a similar procedure to the proof \ref{prop::mm4} of \cref{prop::mm}, it can be shown that $\|\QQ_{X^{(k+1)}}^\alpha\|_2$ and $\|\nGamma-\nM\|_2$ are bounded as well. Furthermore, as a result of \ref{prop::amm2} of \cref{prop::amm}, if $\nGamma\succ\nM$, we obtain
\begin{equation}\label{eq::dX1}
\|X^{\alpha(k+1)}-X^{\alpha(k)}\|\rightarrow 0
\end{equation}
and
\begin{equation}\label{eq::dX2}
\|X^{\alpha(k)}-X^{\alpha(k-1)}\|\rightarrow 0.
\end{equation}
Substituting \cref{eq::dX1,eq::dX2} into \cref{eq::gradb} and applying that $\|\nGamma-\nM\|_2$, $\|\QQ_{X^{(k+1)}}^\alpha\|_2$ and $\|\gamma^{(k)}\|_2$ are bounded, we obtain
\begin{equation}
\|\grad_\alpha F(X^{(k+1)})\|\rightarrow 0,
\end{equation}
which suggests
\begin{equation}
\nonumber
\grad_\alpha F(X^{(k+1)})\rightarrow 0,
\end{equation}
and furthermore,
\begin{equation}
\nonumber
\grad\, F(X^{(k+1)})\rightarrow \0
\end{equation}
if $\nGamma\succ\nM$. The proof is completed.

\vspace{1.2em}

\noindent\textbf{Proof of \ref{prop::amm4} and \ref{prop::amm5}.\;} Similar to the proof of \ref{prop::mm5} and \ref{prop::mm6} of \cref{prop::mm}, we obtain $\nGamma\succ\nM$ if $\xi > 0$. From \ref{prop::amm2} and \ref{prop::amm3} of \cref{prop::amm}, we further obtain $\|X^{(k+1)}-X^{(k)}\|\rightarrow 0$ and $\grad\, F(X^{(k)})\rightarrow \0$, respectively, if $\xi>0$. The proof is completed.

\renewcommand\thesection{Appendix \Alph{section}}
\numberwithin{equation}{section}
\section{Proof of \cref{prop::chordal}}\label{appendix::D}
\renewcommand\thesection{\Alph{section}}
\vspace{0.75em}

From line~\ref{line::alg3::s} of \cref{algorithm::nag}, it can be seen that \cref{algorithm::nag} implements Nesterov's method \cite{nesterov1983method,nesterov2013introductory} to solve the chordal initialization of \cref{eq::chordal}. Moreover, since \cref{eq::chordal} is a convex optimization problem and Nesterov's method converges quadratically for convex optimization \cite{nesterov1983method,nesterov2013introductory}, it can be concluded that \cref{algorithm::nag}  converges to the global optimum to the distributed chordal initialization with a convergence rate of $O(1/k^2)$. The proof is completed.

\end{document}